\documentclass[a4paper]{amsart}

\usepackage{amsmath,amsthm,amssymb} \usepackage[a4paper,margin=3cm]{geometry}
\usepackage{hyperref}

\newcommand{\id}{\mathord{\operatorname{id}}} \newcommand{\red}{\mathrm{red}}
\newcommand{\Ff}{{\mathcal F}} \newcommand{\RR}{\mathbb{R}}
\newcommand{\NN}{\mathbb{N}} \newcommand{\CC}{\mathbb{C}}
\newcommand{\GG}{\mathbb{G}} \newcommand{\VV}{\mathbb{V}}
\newcommand{\FO}{\mathbb{F}O} \newcommand{\Cst}{$C^*$-\relax}
\newcommand{\ad}{\mathop{\operatorname{ad}}}
\newcommand{\qdim}{\mathop{\operatorname{qdim}}}
\renewcommand{\dim}{\mathop{\operatorname{dim}}}
\newcommand{\Tr}{\mathop{\operatorname{Tr}}}
\newcommand{\Span}{\mathop{\operatorname{Span}}}
\newcommand{\Corep}{\mathop{\operatorname{Corep}}}
\newcommand{\Img}{\mathop{\operatorname{Im}}}
\newcommand{\Ker}{\mathop{\operatorname{Ker}}}
\newcommand{\Irr}{\mathop{\operatorname{Irr}}}
\newcommand{\Ll}{\mathord{\mathcal{L}}}

\newcommand{\ot}{\otimes}
\newcommand{\vnot}{\mathop{\bar\otimes}} \newcommand{\hyph}{{\text-}}
\theoremstyle{plain} \newtheorem{theorem}{Theorem}[section]
\newtheorem{lemma}[theorem]{Lemma}
\newtheorem{proposition}[theorem]{Proposition}
\newtheorem{corollary}[theorem]{Corollary} \newtheorem{claim}{Claim}
\theoremstyle{definition} \newtheorem{definition}[theorem]{Definition}
\newtheorem{example}[theorem]{Example} \newtheorem{remark}[theorem]{Remark}
\newtheoremstyle{forcenum}{\topsep}{\topsep}{\itshape}{}{\bfseries}{.}{5pt plus 1pt minus 1pt}{\thmname{#1}\thmnote{ \bfseries #3}}
\theoremstyle{forcenum}
\newtheorem*{lemman}{Lemma}

\title[Orthogonal free quantum groups]{A cocycle in the adjoint representation
  of \\ the orthogonal free quantum groups} \author{Pierre Fima \and Roland
  Vergnioux}

\begin{document}


\begin{abstract}
  We show that the orthogonal free quantum groups are not inner amenable and we
  construct an explicit proper cocycle weakly contained in the regular
  representation. This strengthens the result of Vaes and the second author,
  showing that the associated von Neumann algebras are full ${\rm II}_1$-factors
  and Brannan's result showing that the orthogonal free quantum groups have
  Haagerup's approximation property.  We also deduce Ozawa-Popa's property
  strong (HH) and give a new proof of Isono's result about strong solidity.
\end{abstract}

\maketitle

\section{Introduction}

The orthogonal free quantum groups $\FO_n$ are given by universal unital \Cst
algebras $A_o(n) = C^*(\FO_n)$ which were introduced by Wang
\cite{Wang_FreeProd} as follows:
\begin{displaymath}
  A_o(n) = C^*\langle v_{ij}, 1\leq i,j\leq n \mid v_{ij} = v_{ij}^* 
  \text{~and~} (v_{ij}) \text{~unitary}\rangle.
\end{displaymath}
Equipped with the coproduct defined by the formula $\Delta(v_{ij}) = \sum
v_{ik}\ot v_{kj}$, this algebra becomes a Woronowicz \Cst algebra
\cite{Woronowicz_Compact} and hence corresponds to a compact and a discrete
quantum group in duality \cite{Woronowicz_Compact,PodlesWoro_Lorentz}, denoted
respectively $O_n^+$ and $\FO_n$.

These \Cst algebras have been extensively studied since their introduction, and
it has been noticed that the discrete quantum groups $\FO_n$ share many
analytical features with the usual free groups $F_n$ from the operator algebraic
point of view. Let us just quote two such results that will be of interest for
this article:
\begin{enumerate}
\item the von Neumann algebras $\Ll(\FO_n)$ associated to $\FO_n$, $n\geq 3$,
  are full $II_1$ factors \cite{VergniouxVaes_Boundary} ;
\item the von Neumann algebras $\Ll(\FO_n)$ have Haagerup's approximation
  property \cite{Brannan_aT}.
\end{enumerate}
On the other hand there is one result that yields an operator algebraic
distinction between $\FO_n$ and $F_n$: it was shown in \cite{Vergnioux_Path}
that the first $L^2$-cohomology group of $\FO_n$ vanishes.

The purpose of this article is to present slight reinforcements and alternate
proofs of the results (1), (2)\ above. Namely we will:
\begin{enumerate}
\item show that the discrete quantum groups $\FO_n$ are not inner amenable (see
  \cite{Murray-vonNeumann_RingsOperators,Effros_InnerAmenability} for non
  abelian free groups);
\item construct an explicit proper cocycle witnessing Haagerup's property
  \cite{DawsFimaSkalskiWhite}.
\end{enumerate}

Moreover, putting these constructions together we will obtain a new result,
namely Property strong (HH) from \cite{OzawaPopa_AtMost1Cartan2}, which is a
strengthening of Haagerup's property and corresponds to the existence of a
proper cocycle in a representation which is weakly contained in the regular
representation. This sheds interesting light on the result from
\cite{Vergnioux_Path} mentioned above, according to which such a proper cocycle
cannot live in a representation strongly contained in the regular representation
(or multiples of it). Note that the cocycle corresponding to the original proof
\cite{Haagerup_RDfree} of Haagerup's property for $F_n$ does live in (multiples
of) the regular representation.

Finally we will mention applications to solidity: indeed the constructions of
the article and Property strong (HH) allow for a new proof of the strong
solidity of $\Ll(\FO_n)$ using the tools of \cite{OzawaPopa_AtMost1Cartan,
  OzawaPopa_AtMost1Cartan2}. Notice that strong solidity has already been
obtained in \cite{Isono_NoCartan} using Property AO$^+$ from
\cite{Vergnioux_Cayley,VergniouxVaes_Boundary} and the tools of
\cite{PopaVaes_UniqueCartanHyper}.

\bigskip \noindent {\bf Acknowledgments.} The second author wishes to thank Uwe
Franz for interesting discussions about the interpretation of
\cite[Theorem~4.5]{Brannan_aT} in terms of conditionally negative type
functions. The first author acknowledges the support of the ANR grants NEUMANN
and OSQPI.

\section{Notation}
\label{sec_notation}

If $\zeta$, $\xi$ are vectors in a Hilbert space $H$, we denote
$\omega_{\zeta,\xi} = (\zeta|\,\cdot\,\xi)$ the associated linear form on
$B(H)$, and we put $\omega_\zeta = \omega_{\zeta,\zeta} \in B(H)^+_*$.

We will use at some places the $q$-numbers $[k]_q = U_{k-1}(q+q^{-1}) =
(q^k-q^{-k})/(q-q^{-1})$, where $(U_k)_{k\in\NN}$ are the dilated Chebyshev
polynomials of the second kind, defined by $U_0 = 1$, $U_1 = X$ and $X U_k =
U_{k-1} + U_{k+1}$.

\bigskip

Following \cite{KustermansVaes_LCQG}, a locally compact quantum group $\GG$ is
given by a Hopf-\Cst algebra $C_0(\GG)$, with coproduct denoted $\Delta$,
satisfying certain axioms including the existence of left and right invariant
weights $\varphi$, $\varphi' : C_0(\GG)_+ \to \left[0,{+\infty}\right]$ :
\begin{displaymath}
  \varphi((\omega\ot\id)\Delta(f)) = \omega(1)\varphi(f) 
  \quad\text{and}\quad
  \varphi'((\id\ot\omega)\Delta(f)) = \omega(1)\varphi'(f),
\end{displaymath}
for all $\omega \in C_0(\GG)_+^*$ and $f \in C_0(\GG)_+$ such that
$\varphi(f)<\infty$. When there is no risk of confusion, we will denote by
$\|f\|_2^2 = \varphi(f^*f) \in\left[0,{+\infty}\right]$ the hilbertian norm
associated with $\varphi$. Using a GNS construction $(H,\Lambda_\varphi)$ for
$\varphi$, one defines the fundamental multiplicative unitary $V \in B(H\ot H)$,
denoted $W$ in \cite{KustermansVaes_LCQG}, by putting
$V^*(\Lambda_\varphi\ot\Lambda_\varphi)(f\ot g) =
(\Lambda_\varphi\ot\Lambda_\varphi)(\Delta(g)(f\ot 1))$. Recall that $\GG$ is
called unimodular, or of Kac type, if $\varphi = \varphi'$.

We will assume for the rest of the article that $\GG$ is discrete, meaning for
instance that there exists a non-zero vector $\xi_0 \in H$ such that
$V(\xi_0\ot\zeta) = \xi_0\ot\zeta$ for all $\zeta\in H$. In that case one can
always identify $C_0(\GG)$ with its image by the GNS representation on $H$,
moreover it can be reconstructed from $V$ as the norm closure of $(\id\ot
B(H)_*)(V)$, and the coproduct is given by the formula $\Delta(f) = V^*(1\ot
f)V$.

From $V$ one can also construct the dual algebra $C^*_\red(\GG)$ as the norm
closure of $(B(H)_*\ot\id)(V)$, equipped with the coproduct $\Delta(x) = V(x\ot
1)V^*$ (opposite to the dual coproduct of \cite{KustermansVaes_LCQG}). Since
$(\omega_{\xi_0}\ot\id)(V) = 1$, this \Cst algebra is unital, and it is in fact
a Woronowicz \Cst algebra \cite{Woronowicz_Compact}. Moreover the restriction of
$\omega_{\xi_0}$ is the Haar state $h$ of $C^*_\red(\GG)$. When there is no risk
of confusion, we will denote $\|x\|_2^2 = h(x^*x)$ the hilbertian norm
associated with $h$. Using $\Lambda_h : x \mapsto x\xi_0$ as a GNS construction
for $h$ we have $(\Lambda_h\ot\Lambda_h)(\Delta(x)(1\ot y)) = \Delta(x) V
(\xi_0\ot y\xi_0) = V (x\xi_0\ot y\xi_0)$, so that $V$ coincides with the
multiplicative unitary associated to $(C^*_\red(\GG),\Delta)$ as constructed in
\cite{BaajSkandalis_UnitMul}.

The von Neumann algebras associated to $\GG$ are denoted $L^\infty(\GG) =
C_0(\GG)''$ and $M = \Ll(\GG) = C^*_\red(\GG)''$, they carry natural extensions
of the respective coproducts and Haar weights.

\bigskip

Recall that a corepresentation $X$ of $V$ is a unitary $X \in B(H\ot H_X)$ such
that $V_{12}X_{13}X_{23} = X_{23}V_{12}$. There exists a so-called maximal \Cst
algebra $C^*(\GG)$ together with a corepresentation $\VV \in M(C_0(\GG)\ot
C^*(\GG))$ such that any corepresentation $X$ corresponds to a unique
$*$-representation $\pi : C^*(\GG) \to B(H_X)$ via the formula $X =
(\id\ot\pi)(\VV)$ \cite[Corollaire~1.6]{BaajSkandalis_UnitMul}. We say that $X$,
$\pi$ are unitary representations of $\GG$. When $X = V$, we obtain the regular
representation $\pi = \lambda : C^*(\GG) \to C^*_\red(\GG)$. Moreover $C^*(\GG)$
admits a unique Woronowicz \Cst algebra structure given by $(\id\ot\Delta)(\VV)
= \VV_{12}\VV_{13}$, and $\lambda$ is then the GNS representation associated
with the Haar state $h$ of $C^*(\GG)$.

On the dual side, a unitary corepresentation of $\GG$, or representation of $V$,
is a unitary $Y \in M(K(H_Y)\ot C^*(\GG))$ such that $(\id\ot\Delta)(Y) = Y_{12}
Y_{13}$, corresponding to a $*$-representation of $C_0(\GG)$ as above. We denote
$\Corep(\GG)$ the category of finite dimensional (f.-d.) unitary
corepresentations of $\GG$, and we choose a complete set $\Irr \GG$ of
representatives of irreducible ones. In the discrete case, it is known that any
corepresentation decomposes in $\Irr \GG$. We have in particular a
$*$-isomorphism $C_0(\GG) \simeq c_0{-}\bigoplus_{r\in\Irr\GG} B(H_r)$ such that
$V$ corresponds to $\bigoplus_{r\in\Irr\GG} r$. We denote $p_r \in C_0(\GG)$ the
minimal central projection corresponding to $B(H_r)$, and $C_c(\GG) \subset
C_0(\GG)$ the algebraic direct sum of the blocks $B(H_r)$. The $*$-algebra
$C_c(\GG)$, together with the restriction of $\Delta$, is a multiplier Hopf
algebra in the sense of \cite{VanDaele_Multiplier}, and we denote its antipode
by $\hat S$.

On the other hand we denote $\CC[\GG]_r \subset C^*(\GG)$ the subspace of
coefficients $(\omega\ot\id)(r)$ of an irreducible corepresentation $r$, and
$\CC[\GG]$ the algebraic direct sum of these subspaces. The coproduct on
$C^*(\GG)$ restricts to an algebraic coproduct on $\CC[\GG]$, which becomes then
a plain Hopf $*$-algebra. Recall also that $\lambda$ restricts to an injective
map on $\CC[\GG]$, hence we shall sometimes consider $\CC[\GG]$ as a subspace of
$C^*_\red(\GG)$. We will denote by $\epsilon$ and $S$ the co-unit and the
antipode of $\CC[\GG]$, and we note that $\epsilon$ extends to $C^*(\GG)$: it is
indeed the $*$-homomorphism corresponding to $X = \id_H\ot \id_\CC$.

\bigskip

We denote $(f_z)_z$ the Woronowicz characters of $C^*(\GG)$, satisfying in
particular $h(xy) = h(y$ ${(f_1*x*f_1)})$, where $\phi * x = (\id\ot\phi)\Delta(x)$
and $x * \phi = (\phi\ot\id)\Delta(x)$. We put as well $F_v = (\id\ot f_1)(v)
\in B(H_v)$ for any f.-d. unitary corepresentation $v \in B(H_v)\ot C^*(\GG)$
and $\qdim v = \Tr F_v$. Denoting $v_{\zeta,\xi} =
(\omega_{\zeta,\xi}\ot\id)(v)$, the Schur orthogonality relations
\cite[(5.15)]{Woronowicz_Matrix} read, for $v$, $w$ irreducible:
\begin{equation} \label{eq_schur_orth} h(v_{\zeta,\xi}^*w_{\zeta',\xi'}) =
  \frac{\delta_{v,w}}{\qdim v} (\zeta'|F_v^{-1}\zeta)(\xi|\xi').
\end{equation}
Note that in Woronowicz' notation we have $v_{kl} = v_{e_k,e_l}$ if $(e_i)_i$ is
a fixed orthonormal basis (ONB) of $H_v$. Similarly, we have the formula
$\varphi(a) = \qdim(v) \Tr(F_va)$ for the left Haar weight $\varphi$ and $a \in
B(H_r) \subset C_0(\GG)$ \cite[(2.13)]{PodlesWoro_Lorentz} --- note that the
coproduct on $C_0(\GG)$ constructed from $V$ in \cite{PodlesWoro_Lorentz} is
opposite to ours, so that $\varphi = h_{dR}$ with their notation.

Finally we will use the following analogue of the Fourier transform:
\begin{displaymath}
  \Ff : C_c(\GG) \to C^*_\red(\GG), ~ 
  a \mapsto (\varphi\ot\id)(V(a\ot 1)).
\end{displaymath}
One can show that it is isometric with respect to the scalar products $(a|b) =
\varphi(a^*b)$, $(x|y) = h(x^*y)$ for $a$, $b \in C_c(\GG)$, $x$, $y \in
C^*_\red(\GG)$ --- see also Section~\ref{sec_weak_reg}.

\bigskip

In this article we will be mainly concerned with orthogonal free quantum groups.
For $Q \in GL_n(\CC)$ such that $Q\bar Q \in \CC I_n$, we denote $A_o(Q)$ the
universal unital \Cst algebra generated by $n^2$ elements $v_{ij}$ subject to
the relations $Q\bar vQ^{-1} = v$ and $v^*v=vv^*=I_n\ot 1$, where $v =
(v_{ij})_{ij} \in M_n(A_o(Q))$ and $\bar v = (v_{ij}^*)_{ij}$. Equipped with the
coproduct $\Delta$ given by $\Delta(v_{ij}) = \sum v_{ik}\ot v_{kj}$, it is a
full Woronowicz \Cst algebra, and we denote $\GG = \FO(Q)$ the associated
discrete quantum group, such that $C^*(\FO(Q)) = A_o(Q)$. In the particular case
$Q = I_n$ we denote $A_o(n) = A_o(I_n)$ and $\FO_n = \FO(I_n)$.

It is known from \cite{Banica_ReprAo} that the irreducible corepresentations of
$\FO(Q)$ can be indexed $v^k$, $k\in\NN$ up to equivalence, in such a way that
$v^0 = \id_\CC\ot 1$, $v^1 = v$, and the following fusion rules hold
\begin{displaymath}
  v^k\ot v^l \simeq v^{|k-l|} \oplus v^{|k-l|+2} \oplus \cdots \oplus v^{k+l}.
\end{displaymath}
Moreover the contragredient of $v^k$ is equivalent to $v^k$ for all $k$. The
classical dimensions $\dim v^k$ satisfy $\dim v^0 = 1$, $\dim v^1 = n$ and
$n\dim v^k = \dim v^{k-1} + \dim v^{k+1}$. We denote $\rho\geq 1$ the greatest
root of $X^2-nX+1$, so that $\dim v^k = [k+1]_\rho$. Similarly we have $\qdim
v^0 = 1$ and $\qdim v^1\qdim v^k = \qdim v^{k-1} + \qdim v^{k+1}$. Moreover, if
$Q$ is normalized in such a way that $Q\bar Q = \pm I_n$, we have $\qdim v =
\Tr(Q^*Q)$. We denote $q \in \left]0,1\right]$ the smallest root of $X^2 -
(\qdim v) X + 1$, so that $\qdim v^k = [k+1]_q$.

\section{The adjoint representation of $\FO_n$}

The aim of this section is to prove that the non-trivial part $\ad^\circ$ of the
adjoint representation of $\FO_n$ factors through the regular
representation. Equivalently, we will prove that all states of the form
$\omega_\xi \circ\ad^\circ$, with $\xi\in H$, factor through $C^*_\red(\FO_n)$.
For this we will use the criterion given by Lemma~\ref{lem_growth} below, which
is analogous to Theorem~3.1 of \cite{Haagerup_RDfree}.

\subsection{Weak containment in the regular representation}
\label{sec_weak_reg}

In this section we gather some useful results which are well-known in the
classical case. We say that an element $f \in L^\infty(\GG)$ is a (normalized)
positive type function if there exists a state $\phi$ on $C^*(\GG)$ such that $f
= (\id\ot\phi)(\VV)$. The associated multiplier is $M = (\id\ot\phi)\Delta :
\CC[\GG] \to \CC[\GG]$. The next lemma shows that it extends to a completely
positive (CP) map on $C^*_\red(\GG)$ characterized by the identity $(\id\ot
M)(V) = V (f\ot 1)$. In the locally compact case, the lemma is covered by
\cite[Theorem~5.2]{Daws_CPMult}. For the convenience of the reader we include a
short and self-contained proof for the discrete case.

\begin{lemma} \label{lem_mult} Let $\phi$ be a unital linear form on $\CC[\GG]$
  and consider $M = (\id\ot\phi)\Delta : \CC[\GG] \to \CC[\GG]$. Then $M$
  extends to a CP map on $C^*_\red(\GG)$ if and only if $\phi$ extends to a
  state of $C^*(\GG)$.
\end{lemma}

\begin{proof}
  By Fell's absorption principle $V_{12}\VV_{13} = \VV_{23}V_{12}\VV_{23}^*$,
  $\Delta$ extends to a $*$-homomorphism $\Delta' : C^*_\red(\GG) \to
  C^*_\red(\GG)\ot C^*(\GG)$. Hence if $\phi$ extends to a state of $C^*(\GG)$,
  $M$ extends to a CP map $M = (\id\ot\phi)\Delta'$ on $C^*_\red(\GG)$. For the
  reverse implication, since $C^*(\GG)$ is the enveloping \Cst algebra of
  $\CC[\GG]$, it suffices to prove that $\phi(xx^*) \geq 0$ for all $x \in
  \CC[\GG]$.

  We choose a corepresentation $v \in B(H_v)\ot C^*_\red(\GG)$. Let $L$ be a GNS
  space for $B(H_v)$, with respect to an arbitrary given state. For $a \in
  B(H_v)$, resp. $\omega \in B(H_v)^*$, we denote $\hat a$, $\hat\omega$ the
  corresponding elements of $L = B(\CC, L)$ resp. $L^* = B(L,\CC)$, and we
  identify $B(H_v)$ with a subspace of $B(L)$ via left multiplication. Similarly
  we denote $\hat v$ the element of $B(\CC,L)\ot C^*_\red(\GG)$ corresponding to
  $v$. We have then $\hat v\hat v^* \in B(L)\ot C^*_\red(\GG)$ and
  $(\hat\omega\ot 1)\hat v\hat v^*(\hat\omega^*\ot 1) = xx^*$ if $x =
  (\omega\ot\id)(v)$.

  If $M$ is CP, the following element of $B(L)\ot C^*_\red(\GG)$ is positive:
  \begin{align*}
    (\id\ot M)(\hat v\hat v^*) &=
    (\id\ot\id\ot\phi)(\id\ot\Delta)(\hat v\hat v^*)\\
    &= (\id\ot\id\ot\phi)(v_{12}\hat v_{13}\hat v_{13}^*v_{12}^*) = v(X\ot
    1)v^*,
  \end{align*}
  where $X = (\id\ot\phi)(\hat v\hat v^*) \in B(L)$.  We conclude that $X$ is
  positive, hence $\phi(xx^*) = \hat\omega X\hat\omega^* \geq 0$ for any
  $x\in\CC[\GG]$.
\end{proof}

\bigskip

We will be particularly interested in the case when $\phi$ factors through
$C^*_\red(\GG)$, or equivalently, when $\phi$ is a weak limit of states of the
form $\sum_{i=1}^n \omega_{\xi_i}\circ\lambda$ with $\xi_i \in H$. In that case
we say that $f$, $\phi$ are weakly associated to $\lambda$, or weakly
$\ell^2$. We have the following classical lemma \cite[Lemma~1]{Fell_Weak}:

\begin{lemma} \label{lem_dense} Let $\pi : C^*(\GG) \to B(K)$ be a
  $*$-representation. Assume that there exists a subset $X \subset K$ such that
  $\overline{\Span} ~ \pi(\CC[\GG])X = K$ and $\omega_\xi\circ\pi$ is weakly
  $\ell^2$ for all $\xi\in X$. Then $\pi$ factors through $C^*_\red(\GG)$.
\end{lemma}

\begin{proof}
  As noted in the original paper of Fell, the proof for groups applies in fact
  to general \Cst algebras, and in particular to discrete quantum groups. More
  precisely, take $A = C^*(\GG)$, $T = \pi$ and $\mathcal{S} = \{\lambda\}$ in
  \cite[Rk~1]{Fell_Weak}.
\end{proof}

On the other hand we say that $\phi \in \CC[\GG]^*$ is an $\ell^2$-form if it is
continuous with respect to the $\ell^2$-norm on $\CC[\GG]$, i.e. there exists
$C\in\RR$ such that $|\phi(x)|^2 \leq C h(x^*x)$ for all $x \in \CC[\GG]$. In
that case we denote $\|\phi\|_2$ the corresponding norm. Clearly, if $\phi$ is
$\ell^2$ then it is weakly $\ell^2$.

Although we will not need this in the remainder of this article, the following
lemma shows that $\ell^2$-forms can also be characterized in terms of the
associated ``functions'' in $C_0(\GG)$ by means of the left Haar weight. Note
that in the unimodular case we have simply $\varphi(ff^*) = \varphi(f^*f) =
\|f\|_2^2$.

\begin{lemma} \label{lem_l2_form} Let $\phi \in \CC[\GG]^*$ be a linear form and
  put $f = (\id\ot\phi)(V)$. Then $\phi$ is an $\ell^2$-form if and only if
  $\varphi(ff^*) < \infty$.
\end{lemma}

\begin{proof}
  We put $p_0 = (\id\ot h)(V) \in C_c(\GG)$, which is also the central support
  of $\hat\epsilon$, and we note the following identity (see the Remark after
  the proof) in the multiplier Hopf algebra $C_c(\GG)$ :
  \begin{equation}
    \label{eq:p0ident}
    \forall f \in C_c(\GG) ~~~
    (\hat S\ot\varphi) (\Delta(p_0)(1\ot f)) = f.
  \end{equation}

  From this we can deduce that the scalar product in $H$ implements, via the
  Fourier transform, the natural duality between $C_c(\GG)$ and $\CC[\GG]$ which
  is given, for $x = (\omega\ot\id)(V) \in\CC[\GG]$ and $f = (\id\ot\phi)(V) \in
  C_c(\GG)$, by $\langle f, x\rangle = (\omega\ot\phi)(V) = \phi(x) =
  \omega(f)$. Indeed we have, using the identities $(\hat S\ot\id)(V) = V^*$ and
  $V_{13}V_{23} = (\Delta\ot\id)(V)$:
  \begin{align*}
    (x\xi_0 | \Ff(f)\xi_0) &=
    h(x^*\Ff(f)) = (\varphi\ot h)((1\ot x^*)V(f\ot 1)) \\
    &= (\bar\omega\hat S\ot\varphi\ot h) (V_{13}V_{23}(1\ot f\ot 1)) =
    (\bar\omega\hat S\ot\varphi)(\Delta(p_0)(1\ot f)) = \bar\omega(f).
  \end{align*}

  This yields, for $x \in \CC[\GG]$ and $\phi$ such that $f=(\id\ot\phi)(V) \in
  C_c(\GG)$, the formula $\phi(x) = \omega(f) = (\Ff(f^*)\xi_0|x\xi_0)$. We get
  in particular $\|\phi\|_2 = \|\Ff(f^*)\|_2 = \|f^*\|_2 =
  \sqrt{\varphi(ff^*)}$.
  
  Now for a general $\phi$ and $r \in \Irr\GG$, denote $\phi_r$ the restriction
  of $\phi$ to $\CC[\GG]_r$.  We have $\|\phi\|_2^2 = \sum_{r\in\Irr\GG}
  \|\phi_r\|_2^2$, $(\id\ot\phi_r)(V) = p_rf$ and $\varphi(ff^*) =
  \sum_{r\in\Irr\GG} \varphi(ff^*p_r)$, so that the identity $\|\phi\|_2^2 =
  \varphi(ff^*)$ still holds in $\left[0,+\infty\right]$.
\end{proof}

\begin{remark}
  The identity~\eqref{eq:p0ident} is equivalent to the fact that the Fourier
  transform $\Ff$ is isometric. Indeed a simple computation yields
  \begin{displaymath}
    h(\Ff(a)^*\Ff(b)) = \varphi[a^* 
    (\hat S\ot\varphi)(\Delta(p_0)(1\ot b))].
  \end{displaymath}
  Note that~\eqref{eq:p0ident} can be easily proved by standard computations in
  $C_c(\GG)$, and this is one convenient way of establishing the fact that $\Ff$
  is isometric.
\end{remark}

\bigskip

Now we consider the case of the discrete quantum group $\FO_n$, whose
irreducible corepresentations $v^k \in B(H_k)\ot A_o(n)$ are labeled, up to
equivalence, by integers $k$. Recall that we denote $\CC[\GG]_k$ the
corresponding coefficient subspaces, and $(U_k)_k$ the series of dilated
Chebyshev polynomials of the second kind. In \cite{Brannan_aT}, Brannan shows
that the multiplier $T_s : C^*_\red(\FO_n) \to C^*_\red(\FO_n)$ associated with
the functional
\begin{displaymath}
  \tau_s : \CC[\FO_n] \to \CC, \quad x\in\CC[\FO_n]_k \mapsto 
  \frac{U_k(s)}{U_k(n)} \epsilon(x)
\end{displaymath}
is a completely positive map as above, for every $s \in \left]2,n\right]$ ---
see also Section~\ref{sec_deformation}.

Besides it is known that $\FO_n$ satisfies the Property of Rapid Decay. More
precisely, for any $k\in\NN$ and any $x \in \CC[\GG]_k \subset C^*_\red(\GG)$ we
have $\|x\| \leq C (1+k) \|x\|_2$, where $C$ is a constant depending only on $n$
\cite{Vergnioux_RapidDecay}. On the other hand we denote by $l : \CC[\GG] \to
\CC$ the length form on $\FO_n$, whose restriction to $\CC[\GG]_k$ coincides by
definition with $k\epsilon$, and we recall that the convolution of two linear
forms $\phi$, $\psi \in \CC[\GG]^*$ is defined by $\phi * \psi =
(\phi\ot\psi)\Delta$. It is well-known that the convolution exponential $e^\phi
= \sum \phi^{*n}/n!$ is well-defined on $\CC[\GG]$, and in the case of $l$ (or
any other central form) we have simply $e^{-\lambda l}(x) = e^{-\lambda k}
\epsilon(x)$ if $x \in \CC[\GG]_k$.

Using the multipliers $T_s$ and Property RD we can follow the proof of
\cite[Theorem~3.1]{Haagerup_RDfree} and deduce the following lemma:

\begin{lemma} \label{lem_growth} A state $\phi$ on $C^*(\FO_n)$ factors through
  $C^*_\red(\FO_n)$ if and only if $e^{-\lambda l}*\phi$ is an $\ell^2$-form for
  all $\lambda > 0$.
\end{lemma}

\begin{proof}
  We denote by $\phi_k$ the restriction of $\phi$ to the f.-d. subspace
  $\CC[\GG]_k$. Since these subspaces are pairwise orthogonal with respect to
  $h$, we have $\|\phi\|_2^2 = \sum \|\phi_k\|_2^2$. Now for every $s<n$ there
  is a $\lambda>0$ such that $U_k(s)/U_k(n) \leq e^{-\lambda k}$ for all $k$. We
  can then write
  \begin{displaymath}
    \|\tau_s*\phi\|_2^2 = \sum \|(\tau_s*\phi)_k\|_2^2 
    = \sum \frac{U_k(s)^2}{U_k(n)^2} \|\phi_k\|_2^2 
    \leq \sum e^{-2\lambda k}\|\phi_k\|_2^2  = \|e^{-\lambda l}*\phi\|_2^2.
  \end{displaymath}

  Now if $e^{-\lambda l}*\phi$ is $\ell^2$ for all $\lambda>0$, we conclude that
  this is also the case of $\tau_s*\phi$ for all $s \in \left]2,n\right[$. In
  particular $\tau_s*\phi$ is weakly $\ell^2$ for all $s$. We have
  $(\tau_s*\phi)(x) \to \phi(x)$ as $s\to n$ for all $x\in\CC[\GG]$, and
  $\tau_s*\phi$ is a state for all $s$, hence $\tau_s*\phi \to \phi$ weakly and
  it follows that $\phi$ is weakly $\ell^2$.

  Conversely, assume that $\phi$ factors through $C^*_\red(\GG)$. Using Property
  RD we can write for any $x\in \CC[\GG]_k \subset C^*_\red(\GG)$:
  \begin{displaymath}
    |\phi(x)|\leq \|x\| \leq C(1+k) \|x\|_2,
  \end{displaymath}
  hence $\|\phi_k\|_2 \leq C(1+k)$. This clearly implies that $\|e^{-\lambda
    l}*\phi\|_2$ is finite for every $\lambda > 0$.
\end{proof}

\subsection{The adjoint representation}

Recall that we denote by $\lambda : C^*(\GG) \to B(H)$ the GNS representation of
the Haar state $h$, and consider the corresponding right regular representation
$\rho : C^*(\GG) \to B(H)$, $x \mapsto U\lambda(x)U$ given by the unitary
$U(\Lambda_h(x)) = \Lambda_h(f_1*S(x)) = \Lambda_h(S(x*f_{-1}))$. We have in
particular $\rho(v_{ij}) \Lambda_h(x) = \Lambda_h(x (f_1*v_{ji}^*))$, if $v_{ij}
= v_{e_i,e_j}$ are the coefficients of a unitary corepresentation in an
ONB. Recall that $[U C^*_\red(\GG) U, C^*_\red(\GG)] = 0$.

The adjoint representation of $\GG$ is $\ad : C^*(\GG) \to B(H)$, $x \mapsto
\sum \lambda(x_{(1)})\rho(x_{(2)})$. Here $\Delta : x \mapsto \sum x_{(1)}\ot
x_{(2)}$ is the coproduct from $C^*(\GG)$ to $C^*_\red(\GG)\ot_{\max}
C^*_\red(\GG)$. We have $(\id\ot\ad)(v) = (\id\ot\lambda)(v)(\id\ot\rho)(v)$ if
$v$ is a corepresentation of $\GG$. Here are two explicit formulae, for $x \in
\CC[\GG]$ and $v_{ij}$ coefficient of a unitary corepresentation in an ONB:
\begin{align*}
  &\ad(x)\Lambda_h(y) = \sum \Lambda_h(x_{(1)}y(f_1*S(x_{(2)}))), \\
  &\ad(v_{ij})\Lambda_h(y) = \sum \Lambda_h(v_{ik}y(f_1*v_{jk}^*)).
\end{align*}

The corepresentation of $V$ associated to $\rho$ is $W = (1\ot U)V(1\ot U)$, and
the one associated to $\ad$ is $A = VW$. Note that we have, in the notation of
\cite{BaajSkandalis_UnitMul}, $W = \Sigma \tilde V \Sigma$. In particular
\cite[Proposition~6.8]{BaajSkandalis_UnitMul} shows that $W(1\ot f)W^* =
\sigma\Delta(f)$ for $f\in C_0(\GG)$. This means that the multiplicative unitary
$W^*$ is associated to the discrete quantum group $\GG^{\textrm{co-op}}$.

\begin{lemma}
  The canonical line $\CC\xi_0 \subset H$ is invariant for the adjoint
  representation $\ad : C^*(\GG) \to B(H)$ if and only if $\GG$ is
  unimodular. In that case, $\xi_0$ is a fixed vector for $\ad$.
\end{lemma}

\begin{proof} \label{lem_inv_line} For $x \in \CC[\GG]$ one can compute
  $\ad(x)\xi_0 = \sum x_{(1)}(f_1*S(x_{(2)}))\xi_0$, which equals in the
  unimodular case $\sum x_{(1)}S(x_{(2)})\xi_0 = \epsilon(x)\xi_0$. So in that
  case $\xi_0$ is fixed.

  Furthermore, for $v\in\Irr\GG$ the previous computation together with
  Woronowicz' orthogonality relations lead to
  \begin{align*}
    (\xi_0 | \ad(v_{ij}) \xi_0) &= \sum_k h(v_{ik}(f_1*v_{jk}^*))
    = \sum_k h((v_{jk}^**f_{-1}) v_{ik}) \\
    &= \sum_k \frac{\delta_{ji}\delta_{kk}}{\qdim v} = \frac{\dim v}{\qdim v}
    \epsilon(v_{ij}).
  \end{align*}
  In particular, denoting $\phi : x \mapsto (\xi_0|\ad(x)\xi_0)$, we see that
  $(\id\ot\phi)(v) = (\dim v/\qdim v) \id$. Now if the line $\CC\xi_0$ is
  invariant, $\phi$ is a character and since $v$ is unitary we must have $\dim
  v=\qdim v$ for all $v \in \Irr\GG$. This happens exactly when $\GG$ is
  unimodular.
\end{proof}

\begin{definition}
  If $\GG$ is a unimodular discrete quantum group, we denote $\ad^\circ$ the
  restriction of the adjoint representation of $C^*(\GG)$ to the subspace
  $H^\circ = \xi_0^\bot \subset H$.
\end{definition}

Our aim in the remainder of the section is to show that $\ad^\circ$ factors
through $\lambda$ in the case of $\GG = \FO_n$, $n\geq 3$ --- or more generally
for unimodular orthogonal free quantum groups. We will need some estimates
involving coefficients of irreducible corepresentations of $\GG$.

\bigskip

In order to carry on the computations, we will need to be more precise about the
irreducible corepresentations of $\FO(Q)$. We denote by $v = v^1 \in B(\CC^n)\ot
A_o(Q)$ the fundamental corepresentation. Then we introduce recursively the
irreducible corepresentation $v^k \in B(H_k) \ot A_o(Q)$ as the unique
subcorepresentation of $v^{\ot k}$ not equivalent to any $v^l$, $l<k$. In this
way we have $H_k \subset H_1^{\ot k}$, with $H_0 = \CC$ and $H_1 = \CC^n$. We
denote by $P_k \in B(H_1^{\ot k})$ the orthogonal projection onto $H_k$. Recall
that each irreducible corepresentation of $A_o(Q)$ is equivalent to exactly one
$v^k$, and that we have the equivalence of corepresentations
\begin{displaymath}
  H_k\ot H_l \simeq H_{|k-l|} \oplus H_{|k-l|+2} \oplus \cdots 
  \oplus H_{k+l-2} \oplus H_{k+l}.
\end{displaymath}
It is known that $\dim H_k = [k+1]_\rho =
(\rho^{k+1}-\rho^{-k-1})/(\rho-\rho^{-1})$. Note in particular that we have
$D_1\rho^k\leq\dim H_k\leq D_2\rho^k$ with constants $0<D_1<D_2$ depending only
on $n$.

\bigskip

Let us also denote by $Q^k_r \in L(H_1^{\ot k})$ the orthogonal projection onto
the sum of all subspaces equivalent to $H_r$, so that $P_k = Q^k_k$. If $p_r$ is
the minimal central projection of $C_0(\FO_n)$ associated with $v^r$, then
$Q^k_r$ also corresponds to $\Delta^{k-1}(p_r) $ via the natural action of
$C_0(\FO_n)$ on $H_1$.

Note that the subspace $\Img Q^{a+b+c}_r(P_{a+b}\ot P_c)$ (resp. $\Img
Q^{a+b+c}_r(P_a\ot P_{b+c}))$ corresponds to the unique subcorepresentation of
$v^{a+b}\ot v^c$ (resp. $v^a\ot v^{b+c}$) equivalent to $v^r$, when it is
non-zero. When $r=a+b+c$ both spaces coincide with $H_{a+b+c}$. On the other
hand one can show that these subspaces of $H_1^{\ot a+b+c}$ are pairwise ``far
from each other'' when $r<a+b+c$ and $b$ is big. More precisely, Lemma~A.4 of
\cite{VergniouxVaes_Boundary} shows that
\begin{displaymath}
  \|(P_{a+b}\ot P_c)Q^{a+b+c}_r (P_a\ot P_{b+c})\| \leq C_1 q^b
\end{displaymath}
for some constant $C_1 > 0$ depending only on $q$. Indeed when $r$ varies up to
$a+b+c-2$ the maps on the left-hand side live in pairwise orthogonal subspaces
of $H_1^{\ot a+b+c}$ and sum up to $(P_{a+b}\ot P_c) (P_a\ot P_{b+c}) -
P_{a+b+c}$.

\bigskip

Since we are interested in $\ad^\circ$, we assume in the remainder of this
section that $\FO(Q)$ is unimodular : equivalently, $Q$ is a scalar multiple of
a unitary matrix, or $q\rho=1$.

In the following lemma we give an upper estimate for
$|(\omega_\zeta\ot\Tr_k)((P_l\ot P_k)Q^{k+l}_r\Sigma_{lk})|$, where $\zeta$ is
any vector in $H_l$, $\Tr_k$ is the trace on $B(H_k)$, and $\Sigma_{lk} : H_l\ot
H_k \to H_k\ot H_l$ is the flip map. Note that $\|(P_l\ot P_k) Q^{k+l}_r
\Sigma_{lk}\| \leq 1$ so that $(\dim H_k)\|\zeta\|^2$ is a trivial upper bound,
which grows exponentially with $k$. In the lemma we derive an upper bound which
is polynomial in $k$ (and even constant for $r<k+l$).

Note that it is quite natural to consider maps like $(P_l\ot
P_k)Q^{k+l}_r\Sigma_{lk}$ when studying the adjoint representation of
$A_o(n)$. A non-trivial upper bound for the norm $\|(P_1\ot
P_k)Q^{k+1}_{k-1}\Sigma_{1k}\|$ was given in
\cite[Lemma~7.11]{VergniouxVaes_Boundary}, but it does not imply our tracial
estimate below.

We choose a fixed vector $T_1 = H_1\to H_1$ with norm $\sqrt {\dim H_1}$ ---
$T_1$ is unique up to a phase, and we have $(\id_1\ot T_1^*)(T_1\ot\id_1) = \mu
\id_1$ with $\mu=\pm 1$. We choose then fixed vectors $T_m \in H_m\ot H_m$
inductively by putting $T_m = (P_m\ot P_m)(\id_{m-1}\ot T_1\ot
\id_{m-1})T_{m-1}$. If $(e_s^m)_s$ is an ONB of $H_m$, we have $T_m = \sum
e_s^m\ot \bar e_s^m$ where $(\bar e_s^m)_s$ is again an ONB of $H_m$, and
$\|T_m\| = \sqrt{\dim H_m}$. Finally, we denote more generally
\begin{displaymath}
  T_{ab}^m : H_1^{\ot a-m}\ot H_1^{\ot b-m} \to H_1^{\ot a+b}, ~~
  \zeta\ot\xi \mapsto \sum \zeta\ot e_s^m\ot \bar e_s^m\ot \xi.
\end{displaymath}

\begin{lemma} \label{lem_scalar_1} Let $(e_p^k)_p$ be an ONB of $H_k$. There
  exists a constant $C\geq 1$, depending only on $n$, such that we have, for any
  $0\leq m\leq l\leq k$, $l\neq 0$, $\zeta \in H_l$ and $r = k-l, k-l+2, \ldots,
  k+l-2m-2$:
  \begin{align*}
    \Big|\sum_p(T_{lk}^{m*}(\zeta\ot e_p^k) | 
    Q^{k+l-2m}_rT_{kl}^{m*}(e_p^k\ot\zeta))\Big| 
    &\leq C^l \|\zeta\|^2 \text{,} \\
    \Big|\sum_p(T_{lk}^{m*}(\zeta\ot e_p^k) | 
    P_{k+l-2m}T_{kl}^{m*}(e_p^k\ot\zeta))\Big|
    &\leq k C^l \|\zeta\|^2 \text{~~~ if $l\neq 2m$,} \\
    &\leq (Ck)^l \|\zeta\|^2 \text{~~~ if $l=2m$.} 
  \end{align*}
\end{lemma}

\begin{proof}
  In this proof we denote $d_k = \dim H_k = \qdim v^k$ and $\id_k \in B(H_1^{\ot
    k})$ the identity map. Let us denote by $S^k_r$, $S^k_+$ the sums in the
  statement as well as $S^k = S^k_+ + \sum_r S^k_r$. We first prove the estimate
  for $S^k_r$: using the Lemma~A.4 from \cite{VergniouxVaes_Boundary} recalled
  above, with $a=c=l-m$ and $b = k-l$, we can write
  \begin{align*}
    |S^k_r| &= \Big|\sum_p(T_{lk}^{m*}(\zeta\ot e_p^k)|
    Q^{k+l-2m}_rT_{kl}^{m*}(e_p^k\ot\zeta))\Big | \\
    &\leq \sum_p |(T_{lk}^{m*}(\zeta\ot e_p^k)|
    (P_{l-m}\ot P_{k-m})Q^{k+l-2m}_r(P_{k-m}\ot P_{l-m})T_{kl}^{m*}(e_p^k\ot\zeta))| \\
    &\leq \|(P_{l-m}\ot P_{k-m})Q^{k+l-2m}_r(P_{k-m}\ot P_{l-m})\| \times 
    \|T_m\|^2 \times \sum_p \|\zeta\|^2 \|e_p^k\|^2 \\
    &\leq C_1 q^{k-l} d_k d_m ~ \|\zeta\|^2
    \leq C_1 D_2^2 q^{-l-m} \|\zeta\|^2.
  \end{align*}

  The case of $S_+^k$ is more involved. First notice that $|S_+^k- S^k| = |
  \sum_r S_r^k| \leq (l-m) C_1 D_2^2 q^{-l-m} \|\zeta\|^2$ according to the
  estimate above. This shows that the estimate for $S_+^k$ is equivalent to
  the same estimate for $S^k$ with a possibly different constant. In the case
  when $2m<l$ we will prove the estimate for $S_+^k$ using an induction over
  $k$. We first perform the following transformation using scalar products in
  $H_{l-m}\ot H_{k-l}\ot H_{l-m}$:
  \begin{align*}
    S^k &= \sum_{p,q,i,j} (T_{lk}^{m*}(\zeta\ot e_p^k)| e_i^{l-m}\ot
    e_q^{k-l}\ot e_j^{l-m})
    \times (e_i^{l-m}\ot e_q^{k-l}\ot e_j^{l-m}| T_{kl}^{m*}(e_p^k\ot\zeta)) \\
    &= \sum_{p,q,i,j,s,t} (\zeta\ot e_p^k|e_i^{l-m}\ot e_s^m\ot \bar e_s^m\ot
    e_q^{k-l}\ot e_j^{l-m}) \times \\[-2ex] & \makebox[3cm]{} \times
    (e_i^{l-m}\ot e_q^{k-l}\ot e_t^m\ot \bar e_t^m\ot e_j^{l-m}| e_p^k\ot\zeta) \\
    &= \sum_{q,i,j,s,t} (\zeta|e_i^{l-m}\ot e_s^m)\times (e_i^{l-m}\ot
    e_q^{k-l}\ot e_t^m | P_k(\bar e_s^m\ot e_q^{k-l}\ot e_j^{l-m}))
    \times (\bar e_t^m\ot e_j^{l-m}| \zeta) \\
    &= \sum_{q,s,t} (T_{l,k-l+2m}^{m*}(\zeta\ot \bar e_s^m\ot e_q^{k-l}\ot
    e_t^m) |
    P_kT_{k-l+2m,l}^{m*}(\bar e_s^m\ot e_q^{k-l}\ot e_t^m\ot\zeta)).
  \end{align*}
  One can factor $(P_{k-l+m}\ot\id_{l-m})$ out of $P_k$ and let it move to the
  right of $T_{k-l+2m,l}^{m*}$. Similarly, one can factor $(\id_{l-m}\ot
  P_{k-l+m})$ out, let it move through $T_{k-l+2m,l}^{m*}$ on the left-hand side
  of the scalar product, and take it back to the right-hand side thank to the
  sum over $q$ and $t$ (by cyclicity of the trace). This yields:
  \begin{align*}
    S^k 
    &= \sum_{q,s,t}
    (T_{l,k-l+2m}^{m*}(\zeta\ot \bar e_s^m\ot e_q^{k-l}\ot e_t^m) ~| \\[-2ex]
    & \makebox[3cm]{} |~ P_kT_{k-l+2m,l}^{m*}((P_{k-l+m}\ot\id_m) (\id_m\ot
    P_{k-l+m})(\bar e_s^m\ot e_q^{k-l}\ot e_t^m)\ot\zeta)),
  \end{align*}
  which can be compared to:
  \begin{align*}
    S_+^{k-l+2m} &= \sum_p(T_{l,k-l+2m}^{m*}(\zeta\ot e_p^{k-l+2m})|
    P_kT_{k-l+2m,l}^{m*}(e_p^{k-l+2m}\ot\zeta)) \\
    &= \sum_{q,s,t}(T_{l,k-l+2m}^{m*}(\zeta\ot\bar e_s^m\ot e_q^{k-l}\ot e_t^m)|
    P_kT_{k-l+2m,l}^{m*}(P_{k-l+2m}(\bar e_s^m\ot e_q^{k-l}\ot e_t^m)\ot\zeta)).
  \end{align*}

  More precisely, using again the Lemma~A.4 from \cite{VergniouxVaes_Boundary}
  with $a=c=m$ and $b=k-l$ we get
  \begin{align*}
    |S^k-S_+^{k-l+2m}| &\leq
    \sum_{q,s,t} \|T_m\|^2 \|(P_{k-l+m}\ot\id_m)(\id_m\ot P_{k-l+m})- P_{k-l+2m}\| 
    \|\zeta\|^2 \\
    &\leq \sum_{q,s,t} \|T_m\|^2 ~ C_1 q^{k-l} \|\zeta\|^2 
    \leq C_1 d_m^3 d_{k-l} q^{k-l} \|\zeta\|^2 \leq C_1 D_2^4 q^{-3m}
    \|\zeta\|^2.
  \end{align*}
  Using our previous estimate for the $l-m-1$ terms $S^k_r$ we obtain a
  recursive inequation for $S^k_+$:
  \begin{align*}
    |S^k_+| = |S^k-\sum_r S^k_r| &\leq
    |S^k| + C_1 D_2^2 (l-m) q^{-l-m} \|\zeta\|^2 \\ &\leq
    |S_+^{k-l+2m}| + C_1(D_2^2 l + D_2^4) q^{-3l}\|\zeta\|^2.
  \end{align*}
  For $k\leq 2l$ we use the trivial estimate $|S^k_+| \leq d_k \|T_m\|^2
  \|\zeta\|^2 \leq D_2^2 q^{-3l}\|\zeta\|^2$. Since $l-2m>0$ an easy induction
  on $k$ yields $|S_+^k| \leq klC^l\|\zeta\|^2$, which implies the estimate of
  the statement up to a change of the constant $C$.

  For the case $2m>l$ we consider the vector $\bar\zeta = (\zeta^*\ot\id) T_l$,
  which satisfies $(\id\ot T_m^*)(\zeta\ot\id) = \mu^{l-m}
  (\bar\zeta^*\ot\id)(\id\ot T_{l-m})$ and $(T_m^*\ot\id)(\id\ot\zeta) = \mu^{l-m}
  (\id\ot\bar\zeta^*)(T_{l-m}\ot\id)$. We transform $S^k$ as follows:
  \begin{align*}
    S^k &= \Tr \big((\zeta^*\ot\id_k)(\id_{l-m}\ot T_m\ot\id_{k-m})
    (\id_{k-m}\ot T_m^*\ot\id_{l-m})(\id_k\ot\zeta)\big) \\
    &= \Tr \big((\id_m\ot T_{l-m}^*\ot\id_{k-m})(\bar\zeta\ot\id_{k+l-2m})
    (\id_{k+l-2m}\ot\bar\zeta^*)(\id_{k-m}\ot T_{l-m}\ot\id_m)\big) \\
    &= \Tr \big((\id_{k+l-2m}\ot\bar\zeta^*)(\id_{k-m}\ot T_{l-m}\ot\id_m)
    (\id_m\ot T_{l-m}^*\ot\id_{k-m})(\bar\zeta\ot\id_{k+l-2m})\big), \\
    \bar S^k &= \Tr \big((\bar\zeta^*\ot\id_{k+l-2m})(\id_m\ot T_{l-m}\ot\id_{k-m})
    (\id_{k-m}\ot T_{l-m}^*\ot\id_m)(\id_{k+l-2m}\ot\bar\zeta)\big).
  \end{align*}
  As a result, $\bar S^k$ can be obtained from $S^{k+l-2m}$ by replacing $\zeta$
  with $\bar\zeta$ and $m$ with $l-m$, and we can apply the case $2m<l$ to
  obtain the result --- observe that $k+l-2m \leq k$ and $\|\bar\zeta\| =
  \|\zeta\|$.

  If $l$ is even, we still have to deal with the case $2m=l$,
  which is the subject of the next lemma.
\end{proof}

\begin{lemman}[3.8bis]
  Let $(e_p^k)_p$ be an ONB of $H_k$. There exists a constant $C\geq 1$,
  depending only on $n$, such that we have, for any $\zeta,\xi \in H_{2m}$ and
  $0 < 2m = l \leq k$:
  \begin{displaymath}
    \Big|\sum_p((\id_m\ot T_m^* \ot\id_{k-m})(\zeta\ot e_p^k) | 
    (\id_{k-m}\ot T_m^*\ot\id_m)(e_p^k\ot\xi))\Big|
    \leq (C k)^l  \|\zeta\| \|\xi\|.
  \end{displaymath}
\end{lemman}

\begin{proof}
  We need in fact to prove a more general statement, where $m$ is allowed to
  take two different values $\bar m$, $m \in \NN^*$ on the left and on the right
  of the scalar product, with $\bar m + m \leq k$. For $\zeta \in H_{2\bar m}$,
  $\xi \in H_{2m}$ we denote
  \begin{displaymath}
    S^k(\zeta,\xi) = \sum_p((\id_{\bar m}\ot T_{\bar m}^* \ot\id_{k-\bar m})
    (\zeta\ot e_p^k) | (\id_{k-m}\ot T_m^*\ot\id_m)(e_p^k\ot\xi)).
  \end{displaymath}
  We first transform this quantity by writing the trace of $B(H_k)$ as the
  restriction to a corner of the trace of $B(H_{k-1}\ot H_1)$:
  \begin{align*}
    S^k(\zeta,\xi) &= \sum_{q,s}((\id_{\bar m}\ot T_{\bar m}^*\ot\id_{k-\bar m-1})
    (\zeta\ot e_q^{k-1})\ot e_s^1 | 
    (\id_{k-m}\ot T_m^*\ot\id_m )(P_k(e_q^{k-1}\ot e_s^1)\ot\xi)) \\
    &= \sum_q((\id_{\bar m}\ot T_{\bar m}^*\ot\id_{k-\bar m-1})(\zeta\ot e_q^{k-1}) | 
    (\id_{k-m}\ot
    T_m^*\ot\id_{m-1})(P_k\ot\id_{2m-1})(e_q^{k-1}\ot \tilde \xi)),
  \end{align*}
  where $\tilde\xi = \sum (e_s^1\ot \id_{2m-1}\ot e_s^{1*})(\xi) \in H_1\ot
  H_{2m-1}$ has the same norm as $\xi$. Observe that one has
  $(e_q^{k-1}\ot\tilde\xi) = (\id_k\ot P_m\ot\id_{m-1})(e_q^{k-1}\ot\tilde\xi)$
  so that one can replace $T_m^*$ on the right-hand side of the scalar product
  by $T_{m-1}^*(\id_{m-1}\ot T_1^*\ot\id_{m-1})$ or $T_1^* (\id_1\ot
  T_{m-1}^*\ot \id_1)$.

  Then we use the following generalization of Wenzl's induction formula, which
  can be found e.g. in \cite[Equation (7.4)]{VergniouxVaes_Boundary}:
  \begin{displaymath} 
    P_k = (P_{k-1}\ot\id_1) + \sum_{i=1}^{k-1} (-\mu)^{k-i} 
    \frac{d_{i-1}}{d_{k-1}} (\id_{i-1}\ot T_1\ot \id_{k-i-1}\ot T_1^*)
    (P_{k-1}\ot\id_1).
  \end{displaymath}
  Substituting $P_k$ in the formula for $S^k(\zeta,\xi)$ above, we see that most
  terms vanish. The terms corresponding to $k-m+1\leq i\leq k-1$ vanish because
  in that case both legs of $T_1$ in the factor $(\id_{k-m}\ot
  T_m^*\ot\id_{m-1})(\id_{i-1}\ot T_1\ot \id_{k-i-1}\ot T_1^*\ot \id_{2m-1})$
  hit the left leg of $T_m$ which lies in $H_m$. The terms corresponding to
  $\bar m+1\leq i\leq k-m-1$ vanish because in that case the vector $T_1$ goes
  through $(\id_{k-m}\ot T_m^*\ot\id_{m-1})$ and hits $e_q^{k-1}$ on the left of
  the scalar product.  Finally, the terms corresponding to $1\leq i\leq \bar
  m-1$ vanish because the vector $T_1$ hits $\zeta$ on the left of the scalar
  product.

  Int the term with $(P_{k-1}\ot\id_1)$ we replace $T_m^*$ by
  $T_{m-1}^*(\id_{m-1}\ot T_1^*\ot\id_{m-1})$ on the right-hand side of the
  scalar product and obtain
  \begin{displaymath}
    \sum_q((\id_{\bar m}\ot T_{\bar m}^*\ot\id_{k-\bar m-1})(\zeta\ot e_q^{k-1}) | 
    (\id_{k-m}\ot T_{m-1}^*\ot\id_{m-1})(e_q^{k-1}\ot \xi'))
  \end{displaymath}
  where $\xi' = (T_1^*\ot\id_{2m-2}) (\tilde\xi) = \sum (\bar
  e_s^{1*}\ot\id_{2m-2}\ot e_s^{1*})(\xi) \in H_{2m-2}$. We recognize
  $S^{k-1}(\zeta,\xi')$ and we note that $\|\xi'\| \leq d_1 \|\xi\|$. We remark
  also that when $m=1$ we have $\xi'=0$: indeed in the Kac case $\Sigma_{11}T_1$
  is proportional to $T_1$, and $T_1^*(H_2) = \{0\}$.

  In the term $i = k-m$, the right-hand side of the scalar product reads:
  \begin{align*}
    A &= (\id_{k-m}\ot T_1^*\ot\id_{m-1})(\id_{k-m+1}\ot T_{m-1}^*\ot\id_m) \times \\
    &\makebox[3cm]{} \times(\id_{k-m-1}\ot T_1\ot\id_{3m-2})(\id_{k-2}\ot
    T_1^*\ot\id_{2m-1})(e_q^{k-1}\ot \tilde \xi) \\
    &= \mu (\id_{k-m-1}\ot T_{m-1}^*\ot\id_m) (\id_{k-2}\ot
    T_1^*\ot\id_{2m-1})(e_q^{k-1}\ot \tilde \xi) \\
    &= \mu (\id_{k-m-1}\ot T_m^*\ot\id_m) (e_q^{k-1}\ot \tilde \xi).
  \end{align*}
  Now we decompose $\tilde\xi = \xi'' + Q_{2m-2}^{2m}(\tilde\xi)$, where $\xi''
  = P_{2m}(\tilde\xi)$. Summing over $q$, the terms corresponding to $\xi''$
  yield the quantity $S^{k-1}(\zeta,\xi'')$ and we note that $\|\xi''\| \leq
  \|\xi\|$.

  On the other hand, by Wenzl's recursion formula, $Q_{2m-2}^{2m}(\tilde\xi) =
  \tilde\xi - P_{2m}(\tilde\xi)$ decomposes as a linear combination of vectors
  of the form $(\id_j\ot T_1\ot \id_{2m-2-j})(\xi')$. Again, the contributions
  with $j\neq m-1$ vanish, either because $T_1$ hits the right leg of $T_m^*$,
  or because it hits $e_q^{k-1}$ on the left-hand side of the scalar product (if
  $k>\bar m+m$). The corresponding term in $A$, without the multiplicative
  factor $(-1)^md_{m-1}/d_{2m-1}$, is
  \begin{align*}
    A' &= \mu (\id_{k-m-1}\ot T_m^*\ot\id_m) (\id_{k+m-2}\ot T_1\ot \id_{m-1})
    (e_q^{k-1}\ot\xi') \\
    &= \mu (\id_{k-m-1}\ot T_1^*\ot\id_m) (\id_{k-m}\ot T_{m-1}^*\ot\id_{m+1}) 
    (\id_{k+m-2}\ot T_1\ot \id_{m-1}) (e_q^{k-1}\ot\xi') \\
    &= (\id_{k-m}\ot T_{m-1}^*\ot\id_{m-1}) (e_q^{k-1}\ot\xi'),
  \end{align*}
  and summing the scalar products over $q$ we recognize again
  $S^{k-1}(\zeta,\xi')$.

  Finally for $i = \bar m$ we let the vector $T_1$ from Wenzl's formula go to
  the left-hand side of the tensor product. Noting that $T_m^*(\id_m\ot
  T_1^*\ot\id_m) = T_{m+1}^*$ on $H_{m+1}\ot H_1^{\ot m+1}$ and $T_1^* (\id_1\ot
  T_{\bar m}^*\ot \id_1) = T_{\bar m+1}^*$ on $H_{\bar m+1}\ot H_{\bar m+1}$ we
  obtain
  \begin{displaymath}
    ((\id_{\bar m-1}\ot T_{\bar m+1}^*\ot\id_{k-\bar m-2})(\zeta\ot e_q^{k-1}) | 
    (\id_{k-m-1}\ot T_{m+1}^*\ot\id_{m-2})(e_q^{k-1}\ot \tilde \xi)).
  \end{displaymath}
  In absolute value, the sum over $q$ of these terms is less than 
  \begin{displaymath}
    \dim H_{k-1} \|T_{\bar m+1}\| \|T_{m+1}\| \|\zeta\| \|\xi\| = 
    d_{k-1} \sqrt{d_{\bar m+1} d_{m+1}} \|\zeta\| \|\xi\|.
  \end{displaymath}

  Putting everything together we have obtained 
  \begin{align}
    |S^k(\zeta,\xi)| &\leq |S^{k-1}(\zeta,\xi')| + {\textstyle\frac
      {d_{k-m-1}}{d_{k-1}}} ~ |S^{k-1}(\zeta,\xi'')| + {\textstyle\frac
      {d_{k-m-1}d_{m-1}}{d_{k-1}d_{2m-1}}} ~ |S^{k-1}(\zeta,\xi')| + \nonumber \\
    & \makebox[6cm]{} + {\textstyle\frac {d_{\bar m-1}}{d_{k-1}}} ~
    d_{k-1}\sqrt{d_{\bar m+1} d_{m+1}} ~ \|\zeta\| \|\xi\| \nonumber \\
    & \leq |S^{k-1}(\zeta,\xi'')| + 2 |S^{k-1}(\zeta,\xi')| + D_2^2 q^{-2(m+\bar m)}
    \|\zeta\| \|\xi\|. \nonumber
  \end{align}
  This allows to prove, by induction over $m+\bar m$, that we have
  $|S^k(\zeta,\xi)| \leq D_2^2 (3q^{-2}k)^{m+\bar m} \|\zeta\| \|\xi\|$ for all
  $\zeta \in H_{2\bar m}$, $\xi \in H_{2m}$ and all $k\geq\bar m+m$. We
  initialize the induction at $m+\bar m = 2$, i.e. $m = \bar m = 1$. Then
  $\xi'=0$, and by an easy induction over $k\geq 2$ the inequation above yields
  $|S^k(\zeta,\xi)| \leq D_2^2 q^{-4} k \|\zeta\| \|\xi\|$. Now if the result
  holds at $m+\bar m-1$, and assuming that $m>1$, we apply the induction
  hypothesis to $S^{k-1}(\zeta,\xi')$ in the inequation above and get
  \begin{align*}
    |S^k(\zeta,\xi)| &\leq |S^{k-1}(\zeta,\xi'')| + 
    D_2^2(2 (3q^{-2}(k-1))^{m+\bar m-1} + q^{-2(m+\bar m)}) \|\zeta\| \|\xi\| \\
    & \leq  |S^{k-1}(\zeta,\xi'')| + 
    D_2^2(3q^{-2})^{m+\bar m}(k-1)^{m+\bar m-1} \|\zeta\| \|\xi\|.
  \end{align*}
  The required estimate results then from an easy induction over $k$, using the
  trivial upper bound $d_k\sqrt{d_md_{\bar m}} \|\zeta\| \|\xi\|$ to initialize
  at $k=m+\bar m$. Finally, if $m=1$ but $\bar m > 1$, we proceed similarly but
  transform the left-hand side of the scalar product instead of the right-hand
  side.
\end{proof}

From Lemma~\ref{lem_scalar_1} one can deduce a similar estimate concerning
coefficients of corepresentations. Given ONB's $(e^k_i)_i$, $(e^l_a)_a$ of
$H_k$, $H_l$, we denote by $v^k_{ij}$, $v^l_{ab}$ the associated coefficients of
$v^k$, $v^l$.

\begin{lemma} \label{lem_scalar_2} There exist numbers $C_l$, depending only on
  $n$ and $l$, such that for all $k\geq l > 0$ we have
  \begin{displaymath}
    \sum_{ij} \Big| \sum_p (\Lambda_h(v_{jp}^kv_{ab}^l) |
    \Lambda_h(v_{ab}^lv_{ip}^k))\Big |^2 \leq C_lk^{2l}q^k.
  \end{displaymath}
\end{lemma}

\begin{proof}
  Let us recall some facts from the Woronowicz-Peter-Weyl theory in the Kac
  case. Since $v^k\ot v^l \simeq v^{|k-l|} {\oplus\cdots} \oplus v^{k+l-2}\oplus
  v^{k+l}$, the product of coefficients $v_{jp}^kv_{ab}^l$ decomposes as a sum
  of coefficients of $v^r$, $r = {|k-l|}$, $\ldots$, $k+l-2$, $k+l$. More
  precisely the decomposition corresponds to projecting $e^k_j\ot e^l_a$ and
  $e^k_p\ot e^l_b$ onto the subspace equivalent to $H_r$. Besides, the scalar
  product of two coefficients $(\omega_{x,y}\ot\id)(w)$,
  $(\omega_{x',y'}\ot\id)(w)$ of an irreducible corepresentation $w$ corresponds
  to the scalar product $(x'\ot y|x\ot y')$ up to a factor $(\dim w)^{-1}$,
  according to~\eqref{eq_schur_orth}.

  To compute the products
  $v_{jp}^kv_{ab}^l$, $v_{ab}^lv_{ip}^k$ we need isometric intertwiners, unique
  up to a phase:
  \begin{displaymath}
    \phi_{r}^{k,l} : H_r \to H_k\ot H_l, ~~~~
    \phi_{r}^{l,k} : H_r \to H_l\ot H_k.
  \end{displaymath}
  By irreducibility of $r$, these morphisms are respectively proportional to the
  maps $(P_k\ot P_l)T^m_{kl}P_r$ and ${(P_l\ot P_k)}T^m_{lk}P_r$, where $r =
  (k+l)-2m$ and $H_r \subset (H_{k-m}\ot H_{l-m}) \cap (H_{l-m}\ot H_{k-m})
  \subset H_1^{\ot r}$. These two maps have the same norm $N^{k,l}_m$, which is
  given by the following formula from \cite[Lemma~4.8]{Vergnioux_RapidDecay}:
  \begin{displaymath}
    \big(N^{k,l}_m\big)^2 = 
    \frac{\qdim v^k}{\qdim v^{k-m}} \prod_{q=1}^m 
    \Big(1-\frac{\qdim v^{k-m} \qdim v^{l-m-1}}{\qdim v^{k-q+1}\qdim v^{l-q}}\Big)
  \end{displaymath}
  and is known, again by \cite[Lemma~4.8]{Vergnioux_RapidDecay}, to be
  controlled as follows:
  \begin{displaymath}
    E_1 \frac{\qdim v^k\qdim v^l}{\qdim v^{k+l-2m}} \leq
    \big(N^{k,l}_m\big)^4 \leq
    E_2 \frac{\qdim v^k\qdim v^l}{\qdim v^{k+l-2m}}
  \end{displaymath}
  with $0<E_1<E_2$ independent of $k$, $l$, $m$. In particular, since $m$ takes
  only a finite number of values when $l$ is fixed, and since $(D_1/D_2) q^t
  \leq \qdim v^k/\qdim v^{k+t} \leq (D_2/D_1) q^t$ for all $k$, $t$,
  there exist numbers $0<F_{l,1}<F_{l,2}$ such that $F_{l,1} \leq N^{k,l}_m \leq
  F_{l,2}$ for all $k$, $l$, $m$.
  
  Now we compute the scalar product $(\Lambda_h(v_{jp}^kv_{ab}^l) |
  \Lambda_h(v_{ab}^lv_{ip}^k))$ as the sum over $r=k+l-2m$ of the following
  terms:
  \begin{align*}
    &(\Lambda_h(v_{jp}^kv_{ab}^l) | p_r \Lambda_h(v_{ab}^lv_{ip}^k)) = 
    \\ &\makebox[1cm]{} = \frac 1{\dim H_r} 
    (\phi_r^{l,k*} (e_a^l\ot e_i^k) | \phi_r^{k,l*}(e_j^k\ot e_a^l))
    (\phi_r^{k,l*}(e_p^k\ot e_b^l) | \phi_r^{l,k*} (e_b^l\ot e_p^k)) 
    \\ &\makebox[1cm]{} =
    \frac {1}{(N^{k,l}_m)^2 \dim H_r} 
    (\phi_r^{l,k*} (e_a^l\ot e_i^k) | \phi_r^{k,l*}(e_j^k\ot e_a^l))
    (T_{kl}^{m*}(e_p^k\ot e_b^l) | P_r T_{lk}^{m*}(e_b^l\ot e_p^k))
  \end{align*}
  Denoting $C_{ij}^r$ the sum of these terms over $p$, we can use
  Lemma~\ref{lem_scalar_1} to obtain:
  \begin{align*}
    |C_{ij}^r| & \leq \frac 1{(N^{k,l}_m)^2 \dim H_r} 
    |(\phi_r^{l,k*} (e_a^l\ot e_i^k) | \phi_r^{k,l*}(e_j^k\ot e_a^l))| \times 
    \Big| \sum_p (T_{kl}^{m*}(e_p^k\ot e_b^l) | 
    P_r T_{lk}^{m*}(e_b^l\ot e_p^k)) \Big| \\
    &\leq \frac{(Ck)^l}{(N^{k,l}_m)^2 \dim H_r} 
    |(\phi_r^{l,k*} (e_a^l\ot e_i^k) | \phi_r^{k,l*}(e_j^k\ot e_a^l))|.
  \end{align*}

  Now we sum over $i$ and $j$ to get to the stated estimate. Since $r$ takes
  $l+1$ values we have
  \begin{align*}
    \sum_{ij} \Big|\sum_r C_{ij}^r\Big|^2 &
    \leq (l+1) \sum_{ijr} |C_{ij}^r|^2 \\
    &\leq \frac {(l+1)(Ck)^{2l}}{(N_m^{k,l})^4(\dim H_r)^2} 
    \sum_{ijr} |(\phi^{l,k*}_r(e_a^l\ot e_i^k) | 
    \phi^{k,l*}_r(e_j^k\ot e_a^l))|^2.
  \end{align*}
  We recognize on the right-hand side the squares of the Hilbert-Schmidt norms
  of the maps $(\omega_{e_a^l}\ot\id)$ $(\phi_r^{l,k}\phi_r^{k,l*}\Sigma_{l,k})
  \in B(H_k)$, which are dominated by $\dim H_k$ because the corresponding
  operator norms are less than $1$. Since $r\geq k-l$ we have $(\dim H_k)/(\dim
  H_r)^2 \leq D_2 D_1^{-2} q^{k-2l}$ and we obtain
  \begin{displaymath}
    \sum_{ij} \Big|\sum_r C_{ij}^r\Big|^2 \leq
    \frac {(l+1)^2 (Ck)^{2l}}{(N_m^{k,l})^4} \frac{\dim H_k}{(\dim H_r)^2} 
    \leq ( C^{2l} D_2 D_1^{-2} F_{l,1}^{-4}(l+1)^2 q^{-2l}) \, k^{2l}q^k.
  \end{displaymath}
  \vspace{-4ex}\par
\end{proof}

\begin{theorem} \label{thm_ad_reg} Consider an orthogonal free quantum group
  $\FO(Q)$ which is unimodular.  Then the representation $\ad^\circ$ of
  $C^*(\FO(Q))$ factors through $\lambda$.
\end{theorem}

\begin{proof}
  By Lemma~\ref{lem_dense} it suffices to prove that the states $\phi =
  \omega_\xi\circ\ad : x \mapsto (\xi | \ad(x)\xi)$ on $A_o(n)$ are weakly
  associated to the regular representation, for a set of vectors $\xi$ spanning
  a dense subspace of $H^\circ$. In particular, we can assume that $\xi$ is a
  coefficient of some non-trivial irreducible corepresentation, and we will take
  in fact $\xi = \Lambda_h(v^{m*}_{ab})$.

  We already saw that $\|\phi\|_2^2 = \sum \|\phi_k\|_2^2$, where $\phi_k$ is
  the restriction of $\phi$ to the subspace $\CC[\GG]_k =
  \Span\{v_{ij}^k\}$. Moreover in the unimodular case $(\sqrt{\dim H_k}
  v_{ij}^k)_{ij}$ is an ONB of $\CC[\GG]_k$ with respect to the $\ell^2$ norm,
  and hence $\|\phi_k\|_2^2 = \dim H_k \sum |\phi(v_{ij}^k)|^2$. Using the fact
  that $h$ is a trace, we have
  \begin{align*}
    \phi(v_{ij}^k) &= \sum_p (\Lambda_h(v_{ab}^{m*}) |
    \Lambda_h(v_{ip}^kv_{ab}^{m*}v_{jp}^{k*})) =
    \sum_p (\Lambda_h(v_{ip}^{k*}v_{ab}^{m*}) | \Lambda_h(v_{ab}^{m*}v_{jp}^{k*})) \\
    &= \sum_p (\Lambda_h(v_{jp}^k v_{ab}^m) | \Lambda_h(v_{ab}^m v_{ip}^k)).
  \end{align*}
  Now we can use Lemma~\ref{lem_scalar_2} and we obtain $\|\phi_k\|_2^2 \leq
  \dim(H_k) C_m k^{2m} q^k \leq D_2 C_m k^{2m}$. In particular it is clear now
  that $\|e^{-\lambda l}*\phi\|_2^2 \leq D_2 C_m \sum e^{-2\lambda k}k^{2m} <
  \infty$ for all $\lambda > 0$, and so $\phi$ is weakly associated to the
  regular representation by Lemma~\ref{lem_growth}.
\end{proof}

\subsection{Inner amenability}

The notion of inner amenability for locally compact quantum groups has been
defined in \cite{GhaneiNasr-Isfahani_InnerAmenabilityLCQG}. We will only
consider this notion for discrete quantum groups. Recall that, when $\GG$ is a
discrete quantum group, we denote by $p_0\in L^\infty(\GG)$ the minimal central
projection corresponding to the trivial corepresentation. Following Effros
\cite{Effros_InnerAmenability} we define inner amenability as follows.

\begin{definition}\label{InnerAmenableDQG}
  A discrete quantum group $\GG$ is called \textit{inner amenable} if there
  exists a state $m\in L^\infty(\GG)^*$ such that $m(p_0)=0$ and
  \begin{displaymath}
    m((\id\ot\omega)\Delta(f))=m((\omega\ot\id)\Delta(f))
    \quad\text{for all}\quad
    \omega\in L^\infty(\GG)_*, ~~ f\in L^{\infty}(\GG).
  \end{displaymath}
\end{definition}

\begin{remark}
  Our terminology is different from
  \cite{GhaneiNasr-Isfahani_InnerAmenabilityLCQG} where they call
  \textit{strictly inner amenable} a discrete quantum group satisfying
  Definition \ref{InnerAmenableDQG}. Note however that, according to
  \cite[Remark~3.1(c)]{GhaneiNasr-Isfahani_InnerAmenabilityLCQG}, all discrete
  quantum groups are inner amenable in the sense of
  \cite[Definition~3.1]{GhaneiNasr-Isfahani_InnerAmenabilityLCQG}. In the
  classical case, the following theorem is proved in
  \cite{Effros_InnerAmenability}. The quantum case is more involved and we use
  some techniques from \cite{Tomatsu_AmenableDQG}.
\end{remark}

\begin{theorem} \label{thm_inner_amen} Let $\GG$ be a unimodular discrete
  quantum group. The following are equivalent.
  \begin{enumerate}
  \item $\GG$ is inner amenable.
  \item \label{item_ad_inv} The trivial representation $\epsilon : C^*(\GG) \to
    \CC$ is weakly contained in $\ad^\circ$.
  \end{enumerate}
  Moreover, if $\GG$ is countable and ${\mathcal L}(\GG)$ has property Gamma
  then $\GG$ is inner amenable.
\end{theorem}

\begin{proof}
  $1\Rightarrow 2$. The proof of this implication is similar to the one of the
  implication $1\Rightarrow 2$ in \cite[Theorem~3.8]{Tomatsu_AmenableDQG}, and
  moreover we are in the unimodular case. Let us give a sketch of the
  proof. Putting $A=(\id\ot\ad)(\mathbb{V}) = VW$, it is known that $\epsilon$
  is weakly contained in $\ad^\circ$ if and only if there exists a net of unit
  vectors $\xi_n\in H^\circ = \xi_0^\bot$ such that $\| A(\eta\ot\xi_n) -
  \eta\ot\xi_n\|  \rightarrow 0$ for all $\eta\in H$.

  Let $J_{\varphi}$ and $J_h$ be the modular conjugations of $\varphi$ and $h$
  respectively, with respect to the GNS constructions as in
  Section~\ref{sec_notation}. It is known that $(J_h\ot J_{\varphi})V(J_h\ot
  J_{\varphi})=V^*$, and since $U = J_\varphi J_h = J_h J_\varphi$ in the
  discrete case we also have $(J_h\ot J_{\varphi})W(J_h\ot
  J_{\varphi})=W^*$. Moreover $W = (1\ot U)V(1\ot U) = (J_h\ot J_h) V^* (J_h\ot
  J_h)$, and \cite[Proposition~2.15]{KustermansVaes_VNQG} yields the second
  formula below:
  \begin{equation}\label{MultiplicativeUnitary}
    \begin{aligned}
      V^*(\Lambda_{\varphi}(g)\ot\Lambda_{\varphi}(f)) &=
      (\Lambda_{\varphi}\ot\Lambda_{\varphi})(\Delta(f)(g\ot 1)) \,\,\text{and}
      \\ W(\Lambda_{\varphi}(g)\ot\Lambda_{\varphi}(f)) &=
      (\Lambda_{\varphi}\ot\Lambda_{\varphi})(\sigma\Delta(f)(g\ot 1)),
    \end{aligned}
  \end{equation}
  for all $f,g\in\mathcal{N}_{\varphi}$, and $\sigma$ the flip map on
  $L^\infty(\GG)\vnot L^\infty(\GG)$.  Note that for the second formula to hold
  in the non-unimodular case one has to replace $\Lambda_\varphi$ by a GNS
  construction for the right Haar weight $\varphi'$.

  We use also the identification $C_0(\GG) = \bigoplus_\alpha B(H_\alpha)$,
  where $\alpha$ runs over $\Irr\GG$, and we recall that $\varphi = \sum_\alpha
  \dim(H_\alpha)\Tr_\alpha$ in this identification. We denote
  $(e_{ij}^\alpha)_{ij}$ the matrix units of $B(H_\alpha)$ associated to a
  chosen ONB $(e_i^\alpha)_i$ of $H_\alpha$.  Recall finally the notation
  $\omega_{\zeta,\xi} \in B(H)_*$ for $\zeta$, $\xi \in H$. In what follows we
  restrict these linear forms to $L^\infty(\GG)$ and we convolve them according
  to $\Delta$. Now we have:
       
  \begin{claim} \label{claim_inner_1} Let $f=f^*\in \mathcal{N}_{\varphi}
    \subset L^\infty(\GG)$ and put
    $X=\sigma\Delta(f^2)-\Delta(f^2)=\sum_{\alpha,i,j}e^{\alpha}_{ij}\ot
    X(\alpha)_{ ij}$. For all $\alpha\in\Irr(\GG)$ we have:
    \begin{displaymath}
      \| \omega_{\Lambda_{\varphi}(f)}*
      \omega_{J_h\Lambda_{\varphi}(e^{\alpha}_{i1}),J_h\Lambda_{\varphi}(e^{\alpha}_{j1})}-
      \omega_{J_h\Lambda_{\varphi}(e^{\alpha}_{i1}),J_h\Lambda_{\varphi}(e^{\alpha}_{j1})}*
      \omega_{\Lambda_{\varphi}(f)}\|  
      \geq \varphi(e^{\alpha}_{11})\varphi(|X(\alpha)_{ji}|).
    \end{displaymath}
  \end{claim}

\begin{proof}[Proof of Claim~\ref{claim_inner_1}.]
  Let $z$ be the adjoint of the phase of $X(\alpha)_{ji}$ and write
  \begin{displaymath}
    \begin{split}
      \omega_{\Lambda_{\varphi}(f)}*
      \omega_{J_h\Lambda_{\varphi}(e^{\alpha}_{i1}),J_h\Lambda_{\varphi}(e^{\alpha}_{j1})}(z)-
      \omega_{J_h\Lambda_{\varphi}(e^{\alpha}_{i1}),J_h\Lambda_{\varphi}(e^{\alpha}_{j1})}*
      \omega_{\Lambda_{\varphi}(f)}(z)\\
      =\left(\omega_{J_h\Lambda_{\varphi}(e^{\alpha}_{i1}),J_h\Lambda_{\varphi}(e^{\alpha}_{j1})}
        \ot\omega_{\Lambda_{\varphi}(f)}\right)(\sigma\Delta(z)-\Delta(z)).
    \end{split}
  \end{displaymath}
  Using the formulas $\sigma\Delta(z)=W(1\ot z)W^*$, $\Delta(z)=V^*(1\ot z)V$,
  $(J_h\ot J_{\varphi})V(J_h\ot J_{\varphi})=V^*$, $(J_h\ot J_{\varphi})W(J_h\ot
  J_{\varphi})=W^*$ and Equations~(\ref{MultiplicativeUnitary}), one obtains, as
  in the proof of \cite[Lemma 3.14]{Tomatsu_AmenableDQG}, the formula:
  \begin{displaymath}
    \omega_{\Lambda_{\varphi}(f)}*
    \omega_{J_h\Lambda_{\varphi}(e^{\alpha}_{i1}),J_h\Lambda_{\varphi}(e^{\alpha}_{j1})}(z)-
    \omega_{J_h\Lambda_{\varphi}(e^{\alpha}_{i1}),J_h\Lambda_{\varphi}(e^{\alpha}_{j1})}*
    \omega_{\Lambda_{\varphi}(f)}(z)=
    \varphi(e^{\alpha}_{11}) \varphi(|X(\alpha)_{ji}|).
  \end{displaymath}
  Since $z$ is a partial isometry, the result follows.
\end{proof}

\begin{claim} \label{claim_inner_2} Let $f\in\mathcal{N}_{\varphi}\cap
  L^\infty(\GG)_+$. Then, for all $\alpha\in\Irr(\GG)$, one has:
  \begin{displaymath}
    \| A(\Lambda_{\varphi}(e^{\alpha}_{ij})\ot\Lambda_{\varphi}(f))-
    \Lambda_{\varphi}(e^{\alpha}_{ij})\ot\Lambda_{\varphi}(f)\| ^2\leq
    (\varphi\ot\varphi)(|X(\alpha)|).
  \end{displaymath}
\end{claim}

\begin{proof}[Proof of Claim~\ref{claim_inner_2}.]
  The proof is the same as \cite[Lemma~3.16]{Tomatsu_AmenableDQG}. Define
  $Y=\sigma\Delta(f)-\Delta(f)$. By Equations (\ref{MultiplicativeUnitary}) the
  left hand side of the inequality of the Claim is
  \begin{displaymath}
    \begin{split}
    \|(\Lambda_{\varphi}\ot\Lambda_\varphi)(\sigma\Delta(f)(e^\alpha_{ij}\ot1))-
    (\Lambda_{\varphi}\ot\Lambda_\varphi)(\Delta(f)(e^\alpha_{ij}\ot1))\|^2
    =\| (\Lambda_{\varphi}\ot\Lambda_\varphi)(Y(e^\alpha_{ij}\ot1))\|^2\\
    =(\varphi\ot\varphi)((e^\alpha_{ji}\ot1)Y^2(e^\alpha_{ij}\ot 1))
    =(\varphi\ot\varphi)(Y^2(e^\alpha_{ii}\ot 1))
    =(\varphi\ot\varphi)(Y(e^\alpha_{ii}\ot 1)Y).
    \end{split}
  \end{displaymath}
  Let $p_\alpha$ be the minimal central projection in $L^\infty(\GG)$
  corresponding to $\alpha$. Since $e^\alpha_{ii}\leq p_\alpha$ and by using the
  Powers-St\o rmer inequality we get
  \begin{align*}
    (\varphi\ot\varphi)(Y(e^\alpha_{ii}\ot 1)Y)&\leq
    (\varphi\ot\varphi)(Y(p_\alpha\ot 1)Y)
    = \|(\Lambda_{\varphi}\ot\Lambda_\varphi)(\sigma\Delta(f)(p_\alpha\ot 1)-
    \Delta(f)(p_\alpha\ot 1))\|^2\\
    &\leq \|\omega_{(\Lambda_{\varphi}\ot\Lambda_\varphi)(\sigma\Delta(f)(p_\alpha\ot 1))}-
    \omega_{(\Lambda_{\varphi}\ot\Lambda_\varphi)(\Delta(f)(p_\alpha\ot 1))}\|
    =(\varphi\ot\varphi)(|X(\alpha)|).
  \end{align*}
  This concludes the proof of Claim~\ref{claim_inner_2}.
\end{proof}

We can now finish the proof of $1\Rightarrow 2$. Let $m\in L^\infty(\GG)^*$ be a
state such that $m(p_0)=0$ and
$m((\id\ot\omega)\Delta(f))=m((\omega\ot\id)\Delta(f))$ for all $\omega\in
L^\infty(\GG)_*$, $f\in L^\infty(\GG)$. By the weak* density of the normal
states $\omega\in L^\infty(\GG)_*$ such that $\omega(p_0)=0$ in the set of
states $\mu\in L^\infty(\GG)^*$ such that $\mu(p_0)=0$, there exists a net of
normal states $(\omega_n)_n$ such that $\omega_n(p_0)=0$ for all $n$ and
$\omega_n*\omega-\omega*\omega_n\rightarrow 0$ weak* for all $\omega\in
L^\infty(\GG)_*$.

By the standard convexity argument, we may and will assume that
$\| \omega_n*\omega-\omega*\omega_n\| \rightarrow 0$. Now, since $L^\infty(\GG)$
is standardly represented on $H$, and by a straightforward cut-off argument, we
can assume that $\omega_n = \omega_{\Lambda_\varphi(f_n)}$ with $f_n \in
L^\infty(\GG)_+ \cap C_c(\GG)$, $p_0 f_n = 0$ and $\|f_n\|_2 = 1$. Applying
Claim~\ref{claim_inner_1} (instead of \cite[Lemma~3.14]{Tomatsu_AmenableDQG}),
\cite[Lemma~3.15]{Tomatsu_AmenableDQG} and Claim~\ref{claim_inner_2} (instead of
\cite[Lemma~3.16]{Tomatsu_AmenableDQG}), we obtain
\begin{displaymath}
  \| A(\Lambda_{\varphi}(e^{\alpha}_{ij})\ot\Lambda_{\varphi}(f_n))-
  \Lambda_{\varphi}(e^{\alpha}_{ij})\ot\Lambda_{\varphi}(f_n)\|   \to 0
\end{displaymath}
for all $\alpha$, $i$, $j$. Putting $\xi_n = \Lambda_\varphi(f_n) \in H^\circ$,
this easily implies that $\| A(\eta\ot\xi_n)- \eta\ot\xi_n\|  \to 0$ for all
$\eta\in H$, hence $\epsilon$ factors through $\ad^\circ$.

\medskip

\noindent$2\Rightarrow 1$. Take a net of unit vectors $\xi_n\in H^\circ$ such
that $\| A(\eta\ot\xi_n) - \eta\ot\xi_n\|  \rightarrow 0$ for all $\eta\in H$, and
put $\widetilde{\xi}_n=J_{\varphi}\xi_n$. Then $\widetilde{\xi}_n$ has norm one,
it is orthogonal to $\xi_0$ for all $n$ and we have for all $\eta\in H$:
\begin{align*}
  \| W^*(\eta\ot\widetilde{\xi}_n)-V(\eta\ot\widetilde{\xi}_n)\|  &=
  \| W(J_h\eta\ot\xi_n)-V^*(J_h\eta\ot\xi_n)\|  \\ &=
  \| A(J_h\eta\ot\xi_n)-J_h\eta\ot\xi_n \|  \rightarrow 0.
\end{align*}
Let $m\in L^\infty(\GG)^*$ be a weak* accumulation point of the net of states
$(\omega_{\widetilde{\xi}_n})_n$. One has:
\begin{align*}
  m((\omega_{\eta}\ot\id)\Delta(f))&= \lim\langle(1\ot f)
  V(\eta\ot\widetilde{\xi}_n), V(\eta\ot\widetilde{\xi}_n)\rangle
  =\lim\langle(1\ot f)W^*(\eta\ot\widetilde{\xi}_n),
  W^*(\eta\ot\widetilde{\xi}_n)\rangle\\
  &= m((\omega_{\eta}\ot\id)\sigma\Delta(f))= m((\id\ot\omega_{\eta})\Delta(f))
  \quad\text{for all}\quad\eta\in H,\,\,f\in L^\infty(\GG).
\end{align*}
Moreover, $m(p_0)=\lim\langle p_0\widetilde\xi_n,\widetilde\xi_n\rangle=0$.

\medskip

Finally, suppose that $\GG$ is a countable unimodular discrete quantum group
such that $\Ll(\GG)$ has property Gamma. To show that the co-unit is weakly
contained in $\ad^\circ$ we follow the proof of Effros
\cite{Effros_InnerAmenability}. Write $\Irr(\GG)=\{u^k \mid k\in\NN\}$ where
$u^k\in B(H_k)\ot C^*(\GG)$. Given $n\in\NN$, choose a unitary $u_n\in \Ll(\GG)$
such that $h(u_n)=0$ and
\begin{displaymath}
  \| u_n\lambda(u^{k}_{ij})-\lambda(u^{k}_{ij})u_n\| _2< 
  \frac{1}{n\max\{\dim(u^l)^2 \mid l\leq n\}}
\end{displaymath}
for all $k=1, \ldots, n$ and $1\leq i,j \leq \dim u^k$.  From this inequality,
it is easy to check that, for all $k\leq n$ and all $\eta\in H_k$, one has
\begin{displaymath}
  \| (\id\ot\lambda)(u^k)^*(1\ot u_n^*)(\eta\ot\xi_0)-
  (1\ot u_n^*)(\id\ot\lambda)(u^k)^*(\eta\ot\xi_0)\| <\frac{1}{n}\| \eta\| .
\end{displaymath}
Recall that $\xi_0$ is a fixed vector for $\ad$. Hence, for all $\eta\in H_k$,
one has $(\id\ot\ad)(u^k)(\eta\ot\xi_0)=\eta\ot\xi_0$. Moreover, since the
representations $\lambda$ and $\rho$ commute we find, with $\xi_n=u_n^*\xi_0$,
$k\leq n$ and $\eta\in H_k$,
\begin{align*}
  &\| (\id\ot\ad)(u^k)(\eta\ot\xi_n)-(\eta\ot\xi_n)\|  = \\
  & \makebox[2cm]{}= \| (\id\ot\ad)(u^k)(1\ot
  u_n^*)(\id\ot\ad)(u^k)^*(\eta\ot\xi_0)-
  (\eta\ot\xi_n)\| \\
  & \makebox[2cm]{}= \| (\id\ot\lambda)(u^k)(1\ot
  u_n^*)(\id\ot\lambda)(u^k)^*(\eta\ot\xi_0)-
  (1\ot u_n^*)(\eta\ot\xi_0)\| \\
  & \makebox[2cm]{}<\frac{1}{n}\| \eta\| .
\end{align*}
Since $\xi_n\in H^\circ=\xi_0^\bot$, it follows that the co-unit is weakly
contained in $\ad^\circ$.
\end{proof}

\begin{corollary}[cf. \cite{VergniouxVaes_Boundary}]
  For $n\geq 3$ the discrete quantum group $\FO_n$ is not inner amenable, and in
  particular the von Neumann algebra $\Ll(\FO_n)$ is a full factor.
\end{corollary}

\begin{proof}
  For $n\geq 3$ it is known that $\FO_n$ is not amenable, hence $\lambda$ does
  not weakly contain $\epsilon$. On the other hand by Theorem~\ref{thm_ad_reg}
  the representation $\ad^\circ$ is weakly contained in $\lambda$. Consequently
  $\epsilon$ is not weakly contained in $\ad^\circ$, hence $\FO_n$ is not inner
  amenable and $\Ll(\FO_n)$ is full by Theorem~\ref{thm_inner_amen}.
\end{proof}

\section{Property (HH) for $\Ll(\FO_n)$}
\label{sec_prop_HH}

\subsection{A Deformation}\label{sec_deformation}

Recall that $\Delta : A_o(n)\to A_o(n)\ot A_o(n)$ factors to $\Delta' :
C^*_\red(\FO_n) \to C^*_\red(\FO_n) \ot A_o(n)$ by Fell's absorption
principle. On the other hand, for any element $g \in O_n$ we have a character
$\omega_g : A_o(n) \to \CC$ defined by putting $\omega_g(v_{ij}) = g_{ij}$ and
using the universal property of $A_o(n)$. It is easy to check that
$(\omega_g\ot\omega_h)\Delta = \omega_{gh}$ and that $\omega_e = \epsilon$,
where $e$ is the unit of $O_n$ and $\epsilon$ is the co-unit of $A_o(n)$.

Combining these objects we get $*$-homomorphisms $\alpha_g =
(\id\ot\omega_g)\circ\Delta' : C^*_\red(\FO_n) \to C^*_\red(\FO_n)$ such that
$\alpha_g \circ \alpha_k = \alpha_{gk}$, $\alpha_e = \id$ and $h\circ\alpha_g =
h$. In particular each $\alpha_g$ is an automorphism of $C^*_\red(\FO_n)$ and we
have got an action of $O_n$ on $C^*_\red(\FO_n)$ by trace preserving
automorphisms. Note that there is also an action $\alpha'$ of
$O_n^{\mathrm{op}}$ on $C^*_\red(\FO_n)$ given by $\alpha'_g =
(\omega_g\ot\id)\circ\Delta''$, where $\Delta'' : C^*_\red(\FO_n) \to A_o(n)\ot
C^*_\red(\FO_n)$ is the homomorphism analogous to $\Delta'$.

Let $n\geq 3$. Denote $M=\Ll(\FO_n)$ the von Neumann algebra of $\FO_n$ and
$\tilde M=M\vnot M$. We identify $M$ inside $\tilde M$ via the unital normal
faithful trace preserving $*$-homomorphism $\iota := \Delta$. Denote by $E$ the
trace-preserving conditional expectation from $\tilde M$ to $\iota(M)$. We let
$O_n$ act on $\tilde M$ by putting $A_g = (\alpha_g\ot\id) : \tilde M\to\tilde
M$ for all $g \in O_n$.

\begin{proposition}[cf. \cite{Brannan_aT}] \label{prp_proj_def} Denote $(U_k)_k$
  the dilated Chebyshev polynomials of the second kind and consider, for each $s
  \in \RR$, the densely defined map $T_s : C^*_\red(\FO_n) \to C^*_\red(\FO_n)$
  with domain $\CC[\FO_n]$ such that $T_s = \frac{U_k(s)}{U_k(n)} \id$ on
  $\CC[\FO_n]_k$ for all $k\in\NN$.

  Then for each $g \in O_n$ we have $E \circ A_g \circ \iota = T_s$ where $s =
  \Tr(g)$. In particular, for such $s$ the map $T_s$ extends to a trace
  preserving completely positive map on $C^*_\red(\FO_n)$.
\end{proposition}

\begin{proof}
  For $r\in\NN$, denote $v_{ij}^r$ the coefficients of the $r^{\text{th}}$
  irreducible corepresentation $v^r$ of $\FO_n$, with respect to a given ONB of
  the corresponding space. Using the orthogonality
  relations~\eqref{eq_schur_orth}, it is easy to check that $E(v_{ij}^r\ot
  v_{kl}^s) = \delta_{rs}\delta_{jk} v_{il}^r / U_r(n)$, where $U_r(n) = \dim
  v^r$.
  
  On the other hand, denoting $u^r$ the image of $v^r$ as a representation of
  $O_n$, we have by definition $\omega_g(v^r_{ij}) = u^r_{ij}(g)$. The character
  $\chi_r = \Tr \circ u^r = \sum_k u^r_{kk}$ of $u^r$ is given by $\chi_r(g) =
  U_r(\Tr g)$: indeed by the fusion rule $u^1 \ot u^r \simeq u^{r-1} \oplus
  u^{r+1}$ these characters satisfy the recursion relation of the Chebyshev
  polynomials $U_r$, and we have $\chi_1(g) = \Tr(g)$.  Now it suffices to
  compute:
  \begin{align*}
    E\circ A_g\circ \Delta(v_{ij}^r) = \sum_{k,l} E(v_{ik}^ru_{kl}^r(g)\ot
    v_{lj}^r) = \sum_k v_{ij}^r \frac{u_{kk}^r(g)}{U_r(n)} = v_{ij}^r
    \frac{U_r(\Tr g)}{U_r(n)},
  \end{align*}
  and we recognize the definition of $T_s$ for $s = \Tr(g)$.
\end{proof}

\begin{example} \label{example_def} Define $g_t=\id_{n-2}\oplus R_t\in O_n$
  where
  \begin{displaymath}
    R_t=\left( \begin{array}{cc}
        \cos(t) & -\sin(t) \\ \sin(t) & \cos(t)
      \end{array}\right).
  \end{displaymath}
  Denoting $A_t = A_{g_t}$, we get in this way a $1$-parameter group of
  automorphisms of $\tilde M$ such that $E\circ A_t\circ \iota = T_s$ with
  $s=n-2+2\cos t \in [n-4,n]$.

  Fix $0<t_0<\frac{\pi}{3}$. For all $0<t<t_0$ one has
  $2<1+2\cos(t_0)<\Tr(g_t)<n$. By \cite[Proposition 4.4.1]{Brannan_aT}, there
  exists a constant $C>0$ such that $$\frac{U_k(\Tr(g_t))}{U_k(n)}\leq
  C\left(\frac{\Tr(g_t)}{n}\right)^k$$ for all $0<t<t_0$ and all
  $k\in\NN$. Hence, for $0<t<t_0$, the map $T_s = E\circ A_t\circ\iota$ is
  $L^2$-compact, and this is how Brannan obtains the Haagerup approximation
  property in \cite{Brannan_aT}.

  Finally, consider $S = { 0~1 \choose 1~0}$, $k = \id_{n-2}\oplus S \in O_n$
  and $B = \alpha_k\ot\alpha'_k \in \mathrm{Aut}(\tilde M)$. We clearly have
  $B^2 = \id$, $B A_t = A_{-t}B$, and one can check on the generators $v_{ij}$
  that $B\circ\Delta = \Delta$ so that $B$ restricts to the identity on
  $\iota(M)$. In other words, our deformation can by reversed by an inner
  involution of $\mathrm{Aut}(\tilde M)$ respecting the bimodule structure.
\end{example}

\begin{remark}
  The constructions and results of this Section~\ref{sec_prop_HH} remain valid
  for the other unimodular orthogonal free quantum groups $\FO(Q)$. Up to
  isomorphism, the only missing cases are $\FO(Q_{2n})$ with $Q_{2n} =
  \mathrm{diag}(Q_2, \ldots, Q_2)$ and $Q_2 = {\makebox[1.35ex]{}0~1\, \choose
    -1~0\,}$. But these discrete quantum groups all admit the dual of $SU(2)$ as
  a commutative quotient, and one can build a deformation using the same
  matrices $R_t$ as in Example~\ref{example_def}.

  On the other hand, in the non unimodular case there is no trace preserving
  conditional expectation $E : M\vnot M \to \Delta(M)$, so that the arguments of
  Proposition~\ref{prp_proj_def} do not apply anymore. Recall however that
  Haagerup's Property still holds in that case, as shown in
  \cite{DeCommerFreslonYamashita}.
\end{remark}

Let us give some applications of the preceding construction. It is known that
Haagerup's property is equivalent to the existence of a ``proper conditionally
negative type function'' and/or of a proper cocycle in some
representation. Denoting $\tau_s$ the linear form on $\CC[\FO_n]$ given by
$\tau_s(v_{ij}^r) = \delta_{ij} U_r(s) / U_r(n)$, it follows from the
proposition above and Lemma~\ref{lem_mult} that $\tau_s$ extends to a state of
$C^*(\FO_n)$ for $s\leq n$. Differentiating at $s=n$ we obtain a conditionally
negative form $\psi : \CC[\FO_n] \to \CC$ given by $\psi(v_{ij}^r) = \delta_{ij}
U_r'(n) / U_r(n)$ --- this was also observed in
\cite[Corollary~10.3]{CiprianiFranzKula_LevyProcesses}, where all
$\ad$-invariant conditionally negative forms on $\CC[\FO_n]$ are classified. We
have indeed, for all $x\in \Ker\epsilon \subset \CC[\FO_n]$:
\begin{displaymath}
  \psi(x^*x) = \lim_{s\to n} \frac{\tau_s(x^*x)-\epsilon(x^*x)}{s-n} = \lim_{s\to n} \frac{\tau_s(x^*x)}{s-n} \leq 0.
\end{displaymath}

Moreover the following lemma (cf. also \cite{CiprianiFranzKula_LevyProcesses})
shows that $\psi$ is indeed proper, and more precisely that the associated
function $\Psi = (\id\ot\psi)(V)$ behaves like the ``naive'' length function $L
= \sum r p_r$ at infinity:

\begin{lemma}\label{lem_asympt}
  We have $\frac{U_r'(n)}{U_r(n)} = \frac r{\sqrt{n^2-4}} + O(1)$ as $r \to
  \infty$.
\end{lemma}

\begin{proof}
  We have the well-known formula $U_r(n) = (q^{r+1}-q^{-r-1})/(q-q^{-1})$, where
  $q = \frac 12(n-\sqrt{n^2-4})$.  Differentiating first with respect to $q$, we
  get
  \begin{displaymath}
    \frac{U'_r(n)}{U_r(n)} = \frac{r(1-q^{-2})-2q^{-2}+2q^{-2r-2}+rq^{-2r-2}(1-q^{-2})}
    {(q-q^{-1})(1-q^{-2r-2})} \times \frac{\mathrm{d}q}{\mathrm{d}n}.
  \end{displaymath}
  Noticing that $\mathrm{d}q/\mathrm{d}n = -q / \sqrt{n^2-4}$, we obtain finally
  \begin{displaymath}
    \frac{U'_r(n)}{U_r(n)} = \frac 1{\sqrt{n^2-4}}
    \left(r +  \frac 2{1-q^{-2}} + o(1)\right).
  \end{displaymath}\par
\end{proof}

\subsection{Property (HH)}

It is then natural to ask also for the explicit construction of a proper cocycle
establishing Haagerup's property. Notice that by \cite{Vergnioux_Path} such a
cocycle cannot live in the regular representation or a finite multiple of
it. Now an explicit cocycle can easily be obtained by differentiating the
$1$-parameter group of automorphisms above. More precisely, for any $X \in
\mathfrak{o}_n$, $X\neq 0$, we define $\delta_X = d_XA_g\circ \iota : \CC[\FO_n]
\to \tilde M$. Denoting $u^r \in B(H_r)\ot C(O_n)$ the image of $v^r$, we have
explicitly
\begin{displaymath}
  \delta_X(v_{ij}^r) = \sum_{k,l} v_{ik}^r d_Xu^r_{kl}\ot v_{lj}^r.
\end{displaymath}
Here $d_X u^r_{kl}$ is the differential of $u^r_{kl} : O_n \to \CC$, evaluated
on the tangent vector $X$ at $e$. Since $A_g \circ \iota$ is a $*$-homomorphism,
we get indeed derivations with respect to the $\CC[\FO_n]$-bimodule structure
coming from the embedding $\iota = \Delta$, and in fact we get a Lie algebra map
$(X\mapsto \delta_X)$ from $\mathfrak{o}_n$ to the space of derivations
$\mathrm{Der}(\CC[\FO_n],\tilde M)$.

We moreover denote $f_X = (\id\ot \delta_X)(V)$, which is an unbounded
multiplier of the Hilbert $C_0(\FO_n)$-module $C_0(\FO_n)\ot \tilde M$. Then
$f_X^* f_X$ is an unbounded multiplier of $C_0(\FO_n)$, the ``conditionally
negative'' type function associated to $\delta_X$. The following Lemma shows in
particular that $f_X^* f_X$ corresponds to the conditionally negative form
$\psi$ obtained by differentiating the family of states $\tau_s$.

\begin{lemma}\label{lem_deriv_proper}
  The derivation $\delta_X$ takes its values in the orthogonal complement
  $\tilde M^\circ$ of $\iota(M)$ with respect to $h\ot h$. Moreover we have
  \begin{displaymath}
    f_X^* f_X = \Tr(X^*X) \times (\id\ot \psi)(V) . 
  \end{displaymath}
  In particular for any $X\in \mathfrak{o}_n$, $X\neq 0$, the derivation
  $\delta_X$ is proper in the sense that $p_r f_X^* f_X \geq c_r p_r$ for any
  $r$, with $c_r \to \infty$.
\end{lemma}

\begin{proof}
  The beginning of the proof is quite general and probably well-known to
  experts.
  
  We choose a $1$-parameter subgroup $(g_t)_t$ of $O_n$ such that $g'_0 = X$, we
  put $A_t = A_{g_t}$, $s(t) = \Tr(g_t)$ and we differentiate the identity
  $E\circ A_t \circ\iota = T_{s(t)}$ between linear maps on $\CC[\FO_n]$ from
  Proposition~\ref{prp_proj_def}. We obtain $\Tr(g'_t) \times T'_{s(t)} = E\circ
  A'_t\circ\iota$, where $T'_s$ is the derivative with respect to $s$, and the
  other derivatives are relative to $t$. In particular for $t=0$ this yields
  $E\circ\delta_X = E\circ d_X A_g \circ\iota = \Tr(X) \times T'_n$. Since $\Tr
  X = 0$ for $X \in \mathfrak{o}_n$ we obtain $E\circ\delta_X = 0$, hence
  $\delta_X(\CC[\FO_n]) \subset \tilde M^\circ$.

  Now we differentiate once more at $t=0$, obtaining $\Tr(g''_0) \times T'_n = E
  \circ A''_0 \circ \iota$ as linear maps on $\CC[\FO_n]$. Since $(g_t)_t$ is a
  $1$-parameter group and $X^*+X=0$, we have $g''_0 = g^{\prime 2}_0 = -X^*X$ in
  $M_n(\RR)$. Similarly, since $(A_t)_t$ is a $1$-parameter group of
  trace-preserving automorphisms we have $A''_0 = A^{\prime 2}_0 = -A^{\prime
    *}_0 A'_0$, where $A^{\prime *}_0$ denotes the adjoint of $A'_0$ with
  respect to the hilbertian structure of $\tilde M$. Moreover we have $E =
  \iota^*$ at the hilbertian level, so that $\delta_X^*\delta_X = (A'_0\iota)^*
  (A'_0\iota) = -E\circ A''_0\circ\iota = \Tr(X^*X)\times T'_n$.

  The last identity can also be written $(\delta_X(a)|\delta_X(b)) = \Tr(X^*X)
  h(a^*T'_n(b))$ for all $a$, $b \in \CC[\FO_n]$. As a result we obtain
  \begin{displaymath}
    f_X^* f_X = (\id\ot\delta_X)(V)^* (\id\ot\delta_X)(V)
    = \Tr(X^*X)\times (\id\ot h)(V^*(\id\ot T'_n)(V))
  \end{displaymath}
  as unbounded multipliers --- in other words, the identity above makes sense in
  the f.-d. algebra $p_r C_0(\FO_n)\simeq B(H_r)$ for any $r$. But by definition
  of $\psi$ we have $T'_n = (\id\ot \psi)\circ\Delta$, and using the identity
  $(\id\ot\Delta)(V) = V_{12}V_{13}$ we can write
  \begin{displaymath}
    f_X^* f_X = \Tr(X^*X)\times (\id\ot h\ot \psi)(V^*_{12}V_{12}V_{13}) 
    = \Tr(X^*X)\times (\id\ot \psi)(V).
  \end{displaymath}
  Finally we have $(p_r\ot \psi)(V) = c_r p_r$ with $c_r = U'_r(n)/U_r(n)$ by
  the computation of $\psi$ before Lemma~\ref{lem_asympt}, and the properness
  results from that lemma.
\end{proof}

Recall the construction of the bimodule $K_\pi$ associated to a
$*$-representation $\pi : C^*(\GG) \to B(H_\pi)$. We put $K_\pi = H \ot H_\pi$
where $H$ is the GNS space of the Haar state $h$. The space $K_\pi$ is endowed
with two representations, $\tilde\pi = (\lambda\ot\pi)\Delta' : C^*_\red(\FO_n)
\to B(K_\pi)$ and $\tilde\rho = \lambda^{\mathrm{op}}\ot 1 :
C^*_\red(\FO_n)^{\mathrm{op}} \to B(K_\pi)$, where we put
$\lambda^{\mathrm{op}}(x)\Lambda_h(y) = \Lambda_h(yx)$.

\begin{lemma} \label{lem_bimod} The space $L^2(\tilde M)$, viewed as an
  $M\!\hyph M\!\hyph$bimodule via the left and right actions of $\iota(M) =
  \Delta(M)$, is isomorphic to the $M\!\hyph M\!\hyph$bimodule $K_\pi$ naturally
  associated with the adjoint representation $\pi=\ad$ of $\GG$. The trivial
  part $\CC\xi_0$ of the adjoint representation corresponds to the trivial
  sub-bimodule $\iota(M) \subset \tilde M$.
\end{lemma}

\begin{proof}
  We consider the unitary $V(1\ot U) : L^2(\tilde M) = H\ot H \to H\ot
  H$. Recall that, in the Kac case, $V{(\Lambda_h\ot\Lambda_h)}{(x\ot y)} =
  (\Lambda_h\ot\Lambda_h)(\Delta(x)(1\ot y))$ and $U\Lambda_h(y) =
  \Lambda_h(S(y))$ for $x$, $y\in\CC[\GG]$, so that
  \begin{displaymath}
    V(1\ot U)(\Lambda_h\ot\Lambda_h)(x\ot y) =
    (\Lambda_h\ot\Lambda_h)(\Delta(x)(1\ot S(y))).
  \end{displaymath}
  Recall that $\ad(z)\Lambda_h(y) = \Lambda_h(z_{(1)}y S(z_{(2)}))$.  Then
  the left module structure reads:
  \begin{align*}
    V(1\ot U)(\Lambda_h\ot\Lambda_h)(\Delta(z)(x\ot y)) &=
    (\Lambda_h\ot\Lambda_h) ((z_{(1)}\ot z_{(2)})\Delta(x)
    (1\ot S(y)S(z_{(3)}))) \\
    &= (\lambda\ot\ad)\Delta(z) V(1\ot U) (\Lambda_h\ot\Lambda_h) (x\ot y).
  \end{align*}
  On the other hand we compute:
  \begin{align*}
    V(1\ot U)(\Lambda_h\ot\Lambda_h)((x\ot y)\Delta(z)) &=
    (\Lambda_h\ot\Lambda_h) (\Delta(x)(z_{(1)}\ot z_{(2)})
    (1\ot S(z_{(3)})S(y))) \\
    &= (\Lambda_h\ot\Lambda_h)(\Delta(x)(z\ot S(y))) \\
    &= (\lambda^{\mathrm{op}}(z)\ot\id) V(1\ot U) (\Lambda_h\ot\Lambda_h) (x\ot
    y).
  \end{align*}
  This shows that $V(1\ot U)$ yields an isomorphism $L^2(\tilde M) \simeq
  K_{\ad}$ as $M\hyph M\hyph$bimodules. Moreover, putting $x=y=1$ we see that
  $V(1\ot U)(\Lambda_h\ot\Lambda_h)\iota(M) = \Lambda_h(M)\ot\xi_0$.
\end{proof}

Recall now from \cite{OzawaPopa_AtMost1Cartan2} that a discrete group $G$ has
Property strong (HH) if it admits a proper cocycle with values in a
representation weakly contained in the regular representation. These notions
make sense in the quantum case, and combining the lemma above with
Theorem~\ref{thm_ad_reg} and Lemma~\ref{lem_deriv_proper} we obtain:

\begin{corollary}
  The discrete quantum groups $\FO_n$ have Property strong (HH).
\end{corollary}

\begin{remark}
  Lemma~\ref{lem_bimod} works also at the level of the derivations $\delta_X$
  and yields a formula for the associated cocycles with values in the adjoint
  representation $\pi=\ad$. More precisely, a simple computation shows that
  $V(1\ot U)\delta_X$ is of the form $(\Lambda_h\ot c_X)\Delta$, with $c_X :
  \CC[\FO_n] \to H$ given by $c_X = \Lambda_h m(d_X\alpha_g\ot S)\Delta$, or in
  other terms:
  \begin{displaymath}
    c_X(v_{ij}^r) = \sum (d_X u_{kl}^r) \times \Lambda_h(v_{ik}^rv_{jl}^{r*}).
  \end{displaymath}

  The fact that $V(1\ot U)\delta_X$ is a derivation implies that $c_X$ is a
  cocycle, i.e. $c_X(xy) = \pi(x)c_X(y) + c_X(x)\epsilon(y)$ for $x$, $y \in
  \CC[\FO_n]$. Denoting $g_X = (\id\ot c_X)(V)$ the unbounded multiplier of the
  Hilbert module $C_0(\FO_n)\ot H$ associated with $c_X$, we have $f_X^*f_X =
  g_X^*g_X$. Hence Lemma~\ref{lem_deriv_proper} also shows that the cocycle
  $c_X$ is proper.

  For the generators $v_{ij}$ of $\CC[\FO_n]$ we have $c_X(v_{ij}) = \sum
  X_{kl}\Lambda_h(v_{ik}v_{jl})$. For a particular choice of $X$ one gets
  e.g. $c(v_{ij}) = \Lambda_h(v_{i1}v_{j2} - v_{i2}v_{j1})$, and the other
  values of $c$ can be deduced recursively using the cocycle relation.
\end{remark}

\begin{remark}
  In the classical case, Property (HH) clearly implies the Property
  $\mathcal{QH}_{\mathrm{reg}}$ of \cite{ChifanSinclair}, which is in turn
  equivalent, for exact groups, to bi-exactness and to Property (AO)$^+$, see
  \cite[Proposition~2.7]{PopaVaes_UniqueCartanHyper} and
  \cite[Chapter~15]{BrownOzawa}. In the quantum case, it is proved in
  \cite{Isono_BiExact} that the existence of a ``sufficiently nice'' boundary
  action implies Property (AO)$^+$, but the connection with quasi-cocycles
  remains to be explored. Note that Property (AO)$^+$ for $\FO_n$ is established
  in \cite{Vergnioux_Cayley}.
\end{remark}

\subsection{Strong solidity}

We first recall the following result due to Ozawa, Popa
\cite{OzawaPopa_AtMost1Cartan} and Sinclair \cite{Sinclair_StrongSolidity}.

\begin{theorem} \label{OP07} Let $M$ be a tracial von Neumann algebra which is
  weakly amenable and admits the following deformation property: there exists a
  tracial von Neumann algebra $\tilde M$, a trace preserving inclusion $M\subset
  \tilde M$ and a one-parameter group $(\alpha_t)_{t\in\RR}$ of trace-preserving
  automorphisms of $\tilde M$ such that
  \begin{itemize}
  \item $\lim_{t\rightarrow 0}\| \alpha_t(x)-x\| _2=0$ for all $x\in M$.
  \item $ {\strut}_M\!\left(L^2(\tilde M)\ominus L^2(M)\right)\!{\strut}_M
    \prec\, {\strut}_M\!\left(L^2(M)\ot L^2(M)\right)_{M}$.
  \item $E_M\circ\alpha_t$ is compact on $L^2(M)$ for all $t$ small enough.
  \end{itemize}
  Then for any diffuse amenable von Neumann subalgebra $P\subset M$ we have that
  $\mathcal{N}_{M}(P)''$ is amenable --- in other words $M$ is strongly solid.
\end{theorem}

\begin{proof}
  If $K = L^2(\tilde M)\ominus L^2(M)$ is (strongly) contained in an
  amplification of the coarse bimodule, this is a particular case of
  \cite[Theorem~4.9]{OzawaPopa_AtMost1Cartan}, with $k=1$, $Q = Q_1 = \CC$ and
  $\mathcal{G} = \mathcal{N}_M(P)$. Indeed in that case ``compactness over $Q$''
  for $E_M\circ\alpha_t$ means that its extension to $L^2(M)$ is compact,
  $L^2\langle M, e_{Q_1}\rangle$ is the coarse bimodule $L^2(M)\ot L^2(M)$, so
  that the conclusions of \cite[Proposition~4.8]{OzawaPopa_AtMost1Cartan} hold
  also in our case. On the other hand $P \npreceq_M Q$ means that $P$ is
  diffuse. Moreover if $M$ is weakly amenable and $P$ is amenable, then the
  action of $\mathcal{G} = \mathcal{N}_M(P)$ on $P$ is weakly compact by
  \cite[Theorem~B]{Ozawa_Examples}. Hence the hypotheses of
  \cite[Theorem~4.9]{OzawaPopa_AtMost1Cartan} are satisfied, and we can conclude
  that $N = \mathcal{N}_M(P)''$ is amenable relative to $Q$ inside $M$, which
  just means that $N$ is amenable in our case.

  Note that the proof of \cite[Theorem~4.9]{OzawaPopa_AtMost1Cartan} shows the
  existence, for all non-zero central projection $p \in M$, all $F
  \subset\mathcal{N}_M(P)$ finite and all $\epsilon > 0$ the existence of a
  vector $\zeta \in K\ot L^2(M)$ such that $\|p\zeta\|_2 \geq \|p\|_2/8$,
  $\|[u\ot \bar u,\zeta]\|_2 < \epsilon /2$ for all $u\in F$ and
  $\|x\zeta\|_2\leq \|x\|_2$ for all $x\in M$.  In particular one can then show
  as in \cite[Theorem~3.1]{Sinclair_StrongSolidity} that $K$ is left amenable
  over $\mathcal{N}_M(P)''\subset M$. Now, if $K$ is only weakly contained in
  the coarse bimodule, \cite[Theorem~3.2]{Sinclair_StrongSolidity} still allows
  to conclude that $\mathcal{N}_M(P)''$ is amenable.
\end{proof}

Now by Proposition~\ref{prp_proj_def}, Lemma~\ref{lem_bimod} and Theorem
\ref{thm_ad_reg} one can apply the preceding Theorem to the deformation of
$\Ll(\FO_n)$ presented in Section~\ref{sec_deformation}. Moreover the weak
amenability assumption is satisfied by \cite[Theorem~6.1]{Freslon_CBAP_AO} ---
in fact Freslon also proves that the Cowling-Haagerup constant of $\FO_n$ equals
$1$, so that the invocation of \cite{Ozawa_Examples} in the proof of
Theorem~\ref{OP07} is not necessary in this case. Finally we obtain:

\begin{theorem}
  For all $n\geq 3$, the ${\rm II}_1$ factor $\Ll(\FO_n)$ is strongly solid.
\end{theorem}

Note that this result was already proved in \cite{Isono_NoCartan} using the more
recent approach from \cite{PopaVaes_UniqueCartanHyper} to strong solidity, and
the strong Akemann-Ostrand Property, established in
\cite[Theorem~8.3]{Vergnioux_Cayley} for free quantum groups. Finally, it seems
likely that strong solidity can also be deduced from Property strong (HH) as in
\cite[Corollary~B]{OzawaPopa_AtMost1Cartan2}, by adapting Peterson's techniques
to the proper derivation $\delta_X : \CC[\FO_n] \to \tilde M$ in the quantum
setting. Notice in particular that $\Delta_X = \delta_X^*\delta_X$ is a
``central multiplier'' of $\CC[\FO_n]$ by the proof of
Lemma~\ref{lem_deriv_proper}, so that the passage from the classical to the
quantum setting is probably straightforward.

\bibliographystyle{alpha} \bibliography{propertyHH}

\bigskip

\footnotesize

\noindent {\sc Pierre FIMA \\ \nopagebreak
  Univ Paris Diderot, Sorbonne Paris Cit\'e, IMJ-PRG, UMR 7586, F-75013, Paris, France \\
  Sorbonne Universit\'es, UPMC Paris 06, UMR 7586, IMJ-PRG, F-75005, Paris, France \\
  CNRS, UMR 7586, IMJ-PRG, F-75005, Paris, France \\
  \em E-mail address: \tt pfima@math.jussieu.fr}

\bigskip

\noindent
{\sc Roland VERGNIOUX \\
  Universit\'es de Normandie, France \\
  Universit\'e de Caen, LMNO, 14032 Caen, France \\
  CNRS, UMR 6139, 14032 Caen, France \\
  \em E-mail address: \tt roland.vergnioux@unicaen.fr}

\end{document}